%% file: regime1_final.tex
\documentclass[a4paper, 11pt, reqno]{amsart}

\include{Packages}
\include{Commands}
\usepackage{standalone}
\input{figstyle}

\title{Phase separation in multiply periodic materials with fine microstructures}

\author{Riccardo Cristoferi, Luca Pignatelli}

\begin{document}
\maketitle

\begin{abstract}
    We study a Cahn-Hilliard model for phase separation in composite materials with multiple periodic microstructures.
    These are modeled by considering a highly oscillating potential.
    The focus of this paper is on the case where the scales of the microstructures are smaller than that of phase separation.
    We provide a compactness result and prove that the $\Gamma$-limit of the energy is a multiple of the perimeter.
    In particular, using the recently introduced unfolding operator for multiple scales, we show that taking the limit of all the scales simultaneously is equivalent to taking one limit at a time, starting from the smaller scale and keeping the larger fixed.
\end{abstract}


\section{Introduction}

Composite materials are everywhere in natural (bones, blood, ice, wood) and in synthetic products (foam, colloids, concrete, elastomers).
Therefore, there is a huge interest in understanding how to obtain an effective description of chemical, physical, and mechanical properties that are essential for applications, such as conductivity, stiffness, and permeability.
The complex interactions between the different components result in a dependence of these effective properties on non-trivial details of the microstructure.
In this manuscript, we focus on understanding the distribution of phases at stable equilibrium in a composite material with two periodic microstructures (both larger than the molecular scale) at different scales.\\

For a single material under isothermal conditions, the classical model used to describe stable configurations of phases in a liquid-liquid separation is the celebrated van der Waals--Cahn--Hilliard (also known as Modica--Mortola) functional defined as follows:
    \begin{equation}\label{eq:allen_cahn_iso}
        F_\varepsilon (u) \coloneqq 
            \int_{\Omega} \left[ \frac{1}{\varepsilon} W \!\left(u \right) + \varepsilon | \nabla u |^2 \right] \! \dd x
    \end{equation}
for $u \in W^{1,2}(\Omega; \R^M)$.
Here, $\Omega\subset\R^N$ is an open bounded set with Lipschitz boundary, $\varepsilon>0$ is a small parameter, and the continuous function $W\colon \R^M\to[0,+\infty)$ is the material dependent free energy density with suitable growth at infinity and vanishing at two points (the wells) $a,b\in\R^M$. These latter correspond to the stable phases.
The main goal of the analysis is to understand the asymptotic behavior, as $\varepsilon$ vanishes, of minimizers of $F_\varepsilon$ under a mass constraint of the form
    \begin{equation*}
        \int_{\Omega} u \, \dd x = m \in \R^M.
    \end{equation*}
Using the expansion by $\Gamma$-convergence (see \cite{anzellotti1993asymptotic, DGFra}), it has been proved (see \cite{CarGurSle, ModMor_Esempio, modica1987gradient, Ste_Vect, Ste_Sing, fonseca1989gradient}) that
\begin{equation}\label{eq:latter_functional}
	F_\infty(u)\coloneqq \sigma \, \mathrm{Per}(\{u=a\};\Omega),
\end{equation}
for $u\in \mathrm{BV}(\Omega;\{a,b\})$, namely functions of bounded variation taking values in the set $\{a,b\}$.
Here, the \emph{surface tension} is defined as
\begin{equation}\label{eq:sigma}
\sigma\coloneqq 
    \inf \left\{ \int_{-1}^1 2 \sqrt{W ( \gamma(t))} \, | \gamma' (t)| \, \dd t : \gamma \in \text{Lip}([-1,1]; \R^M ), \gamma(-1)=a, \gamma(1)=b \right\}.
\end{equation}
Note that the limiting energy $F_\infty$ is an anisotropic functional.\\

When modeling phase separation in composite materials, the functional has to adapt in order to take into consideration the heterogeneities of the medium.
We consider the case where the material has a periodic microstructure with scale $\delta>0$ (see Figure \ref{fig:one_scale} to the left).
The natural modification of the functional \eqref{eq:allen_cahn_iso} yields
    \begin{equation}\label{eq:allen_cahn_one}
        F_{\varepsilon,\delta} (u) \coloneqq 
            \int_{\Omega} \left[ \frac{1}{\varepsilon} W \!\left(\cfrac{x}{\delta}, u \right)
                + \varepsilon | \nabla u |^2 \right] \! \dd x. 
    \end{equation}
To model the periodicity we define a periodicity cell $Q_1$, and we require the potential $W\colon Q_1\times \R^M\to[0,\infty)$ to be a Carath\'{e}odory function extended to be $Q_1$-periodic in the first variable, with suitable growth at infinity, and with wells at $a,b\in\R^M$. In particular, the low regularity in the first variable is necessary in order to consider the case of a microstructure with material inclusions $E\subset Q_1$ (see Figure \ref{fig:one_scale} to the right). In such a case, the potential $W$ might jump from one material to another.

\begin{figure}
    \centering
    \input{fig_one_scale.tex}
    \caption{Left: A composite material with a periodic microstructure.
        Right: A microstructure with material inclusions. 
        Different colors correspond to different materials.}
    \label{fig:one_scale}
\end{figure}

Obtaining the asymptotic behavior of the functional $F_{\varepsilon,\delta}$ as $\varepsilon$ and $\delta$ vanish is now more challenging, due to the presence of two parameters in the problem.
Indeed, there is a competition between the process of homogenization of the periodic structure and that of transition between the stable phases. The former happens at the scale $\delta$, while the latter at the (expected) scale $\varepsilon$.
Therefore, the problem naturally gives three regimes:
\[
\delta\ll\varepsilon\quad\quad\quad\quad
\delta\approx\varepsilon\quad\quad\quad\quad
\varepsilon\ll\delta.
\]
In a series of works (see \cite{CriFonGan_fixed_Sup, CriFonHagPop, CriFonGan_fixed_Sub}), the first author with several collaborators prove that the $\Gamma$-limit is of the form
\[
F_\infty(u) \coloneqq \int_{J_u} \sigma(\nu_u)\dd \mathcal{H}^{N-1},
\]
for functions of bounded variations taking only the values $a$ and $b$, namely $u\in \mathrm{BV}(\Omega;\{a,b\})$.
Here, $\nu_u\in\mathbb{S}^{N-1}$ denotes the measure theoretical exterior unit normal to $\{u=a\}$, and the Borel function $\sigma:\S^{N-1}\to[0,\infty)$ depends on the regime under consideration.
Note that, in this case, the limiting energy is expected to be anisotropic, and there are cases where it is truly anisotropic.
In particular, in the former and in the latter regime, it has been proved that the functional $F_\infty$ is the same that would be obtained by taking \emph{one limit at a time}. Namely, in the regime $\delta\ll\varepsilon$, we first send $\delta\to0$ and homogenize the potential, and then send $\varepsilon\to0$.
Similarly, in the regime $\varepsilon\ll\delta$, we first send $\varepsilon\to0$, thus obtaining a surface energy, that we then homogenize by sending $\delta\to0$.\\

The question at the basis of this project, of which this paper is the first step, is the following: what happens when the material has multiple microstructures at different scales? In particular, is it true that the `\textit{principle of multiple scales}' obtained above for one scale and according to which we compute `one limit at a time', holds true also for multiple scales?

\begin{figure}
    \centering
    \input{fig_two_scales.tex}
    \caption{Left: A composite material with nested periodic microstructures.
        Center: A microstructure with material inclusions and a microstructure.
        Right: A microstructure with periodic inclusions.
        Different colors correspond to different materials.}
    \label{fig:two_microscales}
\end{figure}

The prototype of a composite material with multiple microstructures is that of a periodic structure with material inclusions having a periodic (smaller) microstructure as well (see Figure \ref{fig:two_microscales}). Namely, we have in mind potentials $W\colon Q_1\times Q_2\times\R^M\to[0,\infty)$ of the form
    \[
    W(y_1,y_2,z) \coloneqq \ca_{I_1}(y_1) [ \ca_{I_2}(y_2) \Phi_1(z) +  \ca_{Q_2\setminus I_2}(y_2) \Phi_2(z) ] + \ca_{Q_1\setminus I_1}(y_1) \Phi_3(z),
    \]
where $Q_1$ and $Q_2$ are the periodicity cells of the larger and the smaller microstructures. Here $\Phi_1, \Phi_2, \Phi_3$ are classic double-well potentials with the same wells, like the ones of the classical van der Waals--Cahn--Hilliard functional in \eqref{eq:allen_cahn_iso}.
Note that there is no relation between $Q_1$ and $Q_2$.
Therefore, we are naturally led to consider the functional 
\begin{equation}\label{eq:allen_cahn_two}
    F_{\varepsilon,\delta,\eta} (u) \coloneqq \int_\Omega \Bigg[ \frac{1}{\varepsilon} W \Bigg( \!\bigg\{\! \frac{x}{\delta} \! \bigg\}_{Q_1}\!, \bigg\{\! \frac{\delta}{\eta}\!\bigg\{ \!\frac{x}{\delta} \!\bigg\}_{Q_1} \!\bigg\}_{Q_2}\!, u \!\Bigg) + \varepsilon | \nabla u |^2 \Bigg] \! \dd x,
\end{equation}
where $\{ z \}_{Q_i}$ is the fractional part of $z$ with respect to $Q_i$. This notation will be made clear in Definition \ref{def:fractional_decomposition}.

Having the above example in mind, we always assume a separation of scales
\[
\eta\ll\delta.
\]
 This functional, albeit less common in literature, makes the most sense for us, as it naturally has the consequence of being $Q_1$-periodic in the first argument, and $Q_2$-periodic in the second argument. The results of this paper still hold for the classic potential functional given by
$$ G_{\varepsilon,\delta,\eta} (u) \coloneqq \int_\Omega \left[ \frac{1}{\varepsilon} W \left( \frac{x}{\delta}, \frac{x}{\eta}, u \right) + \varepsilon | \nabla u |^2 \right] \dd x, $$
as shown in Remark \ref{remark:alternative_potential}.
A complete study of the $\Gamma$-convergence of $F_{\varepsilon,\delta,\eta}$ requires five distinct regimes:
    \begin{enumerate}\setlength\itemsep{0.2cm}
        \item $\eta \ll \delta \ll \varepsilon$;
        \item $\eta \ll (\delta \approx \varepsilon)$;
        \item $\eta \ll \varepsilon \ll \delta$;
        \item $(\varepsilon \approx \eta) \ll \delta$;
        \item $\varepsilon \ll \eta \ll \delta$.
    \end{enumerate}

In this paper we focus on the regime
\[
\eta \ll \delta \ll \varepsilon,
\]
(see Figure \ref{fig:first_regime}). We prove that the $\Gamma$-limit is the same as the one obtained by taking `one limit at a time': namely, by first sending $\eta$ to zero, then $\delta$ to zero, and finally $\varepsilon$ to zero.
In particular, in our main theorem (see Theorem \ref{thm:gamma_convergence}) we show that
\begin{equation}\label{eq:first_order}
F_\infty(u) \coloneqq \sigma^\mathrm{h} \mathrm{Per}(\{ u=a\}; \Omega),
\end{equation}
for functions $u\in \mathrm{BV}(\Omega;\{a,b\})$, where (compare with \eqref{eq:sigma})
\[
\sigma^\mathrm{h} \coloneqq \inf \left\{ \int_{-1}^1 2 \sqrt{W^\mathrm{h} ( \gamma(t))} \, | \gamma' (t)| \, \dd t : \gamma \in \text{Lip}([-1,1]; \R^M ), \gamma(-1)=a, \gamma(1)=b \right\}, 
\]
and
\[
W^\mathrm{h} (z) \coloneqq \med_{Q_1} \med_{Q_2} W(y_1,y_2,z) \, \dd y_1 \dd y_2.
\]
We also provide a compactness result (see Theorem \ref{thm:compactness}) that allows us to use standard results of $\Gamma$-convergence to prove the convergence of minimizers and minima.
Finally, we consider the case where the functional \eqref{eq:allen_cahn_two} is finite only on configurations satisfying a mass constraint, and we prove that this passes to the limit; namely, that the first order $\Gamma$-limit is finite only for configurations satisfying the mass constraint and, in this case, given by the functional \eqref{eq:first_order} (see Theorem \ref{thm:mass_constraint}).

The main challenge of our analysis is in justifying sending one parameter to zero at a time.
First of all, in the liminf inequality we need new ideas compared to those in \cite[Theorem 4.1]{CriFonGan_fixed_Sup}, where a very convoluted auxiliary potential was used, as we know that this cannot be extended to the case of multiple scales.
The core main novelty of the paper is in developing a new strategy to rigorously perform the liminf inequality (see Proposition \ref{prop:liminf}).
Our idea is to identify a simpler and more manageable auxiliary potential (see Section \ref{subsec:keyestimates} and Section \ref{subsec:stability}) more clearly related to the original potential.
This strategy is significantly easier both in terms of argument and in terms of length.
In the limsup inequality (see Proposition \ref{prop:limsup}) we provide a simpler strategy to show that the recovery sequence for the homogenized potential is also a recovery sequence for the original multiscale functional.
Moreover, when considering the mass constrained problem (see Theorem \ref{thm:mass_constraint}), we are able to consider more general potentials than those currently present in the literature, other than fixing some gaps in those strategies.
Finally, our argument extends, with only notational changes, to any finite number of separated microscales, namely $\delta_1\ll\delta_2\ll\dots \ll\delta_k\ll\varepsilon$.\\

We would like to mention that we are also investigating the other regimes, the results of which will be published in a series of forthcoming papers.\\

\textbf{Other literature.}
Extensions of the functional \eqref{eq:allen_cahn_one} in the context of phase separation in composite materials with one scale of microstructures have been considered by several researchers.
In particular, the case where wells are also dependent on the spatial variable has been investigated by the first author, Fonseca and Ganedi in \cite{CriFonGan_moving_sub} in the regime $\varepsilon\ll\delta$.
Moreover, the case where oscillations are in the singular term has been considered by Ansini, Braides, and Chiad\`{o} Piat in \cite{AnsBraChi2} (see also \cite{AnsBraChi1}), while the effect of a highly oscillating forcing term has been investigated by Dirr, Lucia, and Novaga in \cite{DirLucNov1} and \cite{DirLucNov2}.
Finally, the literature on the stochastic setting features recent contributions by Marziani \cite{Mar}, by Bach, Marziani and Zeppieri \cite{BacMarZep}, by Morfe \cite{Mor}, by Morfe and Wagner \cite{MorWag} and by Donnarumma \cite{Don}.\\

\section{Assumptions and main results}

Let $\Omega\subset\R^N$ be an open bounded set with Lipschitz boundary.
Let $G_1, G_2$ be two subgroups of $\R^N$ with rank $N$, corresponding to the microscales $\delta$ and $\eta$.
Let $Q_1$ be the periodicity cell with respect to $G_1$, and let $Q_2$ be the periodicity cell with respect to $G_2$.
We assume them to be connected, bounded, with Lipschitz boundary, containing the origin, and such that their measure is equal to $1$.
This last assumption is only for the sake of notation in order not to have to write the average integral each time we integrate over $Q_1$ and $Q_2$.
Points in $Q_1, Q_2$ will be denoted by $y_1, y_2$, respectively. \\

Let $W \colon Q_1 \times Q_2 \times \R^M \to [0,+\infty)$ be our potential.
In the results that we will prove, we will need to consider some of the following hypotheses:
\begin{enumerate}\setlength\itemsep{0.2cm}
    \item[\namedlabel{itm:1_H1}{(H1)}] There exist pairwise disjoint open sets $E_1, \ldots, E_K \subset Q_1$ with Lipschitz boundary and $\ln (Q_1 \setminus \bigcup_i E_i ) =0$, and functions $W_i \colon \overline{E_i} \times Q_2 \times \R^M \to [0,+\infty)$, $i = 1, \ldots, K$, such that 
    $$ W(y_1, y_2, z) = \sum_{i=1}^K \mathbbm{1}_{E_i}(y_1) W_i (y_1, y_2, z), $$
    and such that
\begin{itemize}\setlength\itemsep{0.2cm}
    \item[$\circ$] $y_2 \mapsto W_i(y_1, y_2, z)$ is measurable in $Q_2$ for all $(y_1,z) \in \overline{E_i} \times \R^M$;
    \item[$\circ$] there exists a set $\mathcal{N} \subset Q_2$ with $\ln (\mathcal{N}) = 0$ such that, for every $y_2 \in Q_2 \setminus \mathcal{N}$, the map $(y_1, z) \mapsto W_i (y_1,y_2,z)$ is continuous in $\overline{E_i} \times \R^M$.
\end{itemize}
\item[\namedlabel{itm:1_H2}{(H2)}] There exist $a,b \in \R^M$ such that
$$ W(y_1, y_2, z) = 0 \iff z \in \{ a,b \}. $$
\item[\namedlabel{itm:1_H3}{(H3)}] There exists $W_0 \colon  \R^M \to \R$ continuous with $W_0^{-1} (0) = \{a,b\}$, such that
\[
0 \leq W_0 (z) \leq W(y_1, y_2, z) \qquad \forall y_1 \in Q_1, \forall y_2 \in Q_2.
\]
\item[\namedlabel{itm:1_H4}{(H4)}] For every $S > 0$, there exists a constant $C_S > 0$ depending only on $S$ such that
$$ \displaystyle \mathrm{ess\,sup}_{y_1 \in Q_1, y_2 \in Q_2, |z|\leq S} W(y_1,y_2,z) \leq C_S. $$
\item[\namedlabel{itm:1_H5}{(H5)}] There exists $R > 0$ such that for $\ln$-a.e. $y_1 \in Q_1$, $y_2 \in Q_2$, if $| z | \geq R$ then it holds:
$$ W(y_1,y_2,z) \geq \frac{1}{R} |z|. $$
\item[\namedlabel{itm:1_H6}{(H6)}] There exist $r_0 > 0$ and a non-decreasing function $\omega_a \colon [0,r_0] \to [0,+\infty)$ with $\lim_{r \to 0^+} \omega_a (r) = 0$ such that, for $\ln$-a.e. $y_1 \in Q_1$, $y_2 \in Q_2$, it holds
    $$ W(y_1,y_2,z) \leq \omega_a (|z-a|) |z-a| \qquad \text{for all }\, z \in \R^M \text{ with } |z-a| \leq r_0. $$
\end{enumerate}

We would like to comment on each of the above assumption, in order for the reader to have a clear idea of why and where they are needed, as well as to highlight possible extensions.

    \begin{remark}
    The analysis of this paper and the following other four regimes are restricted to the case of two microscales.
    The type of potential that we have in mind is of the form
    \begin{equation}\label{eq:prototype_potential}
        W(y_1,y_2,z) \coloneqq \ca_{I_1}(y_1) [ \ca_{I_2}(y_2) \Phi_1(z) +  \ca_{Q_2\setminus I_2}(y_2) \Phi_2(z) ] + \ca_{Q_1\setminus I_1}(y_1) \Phi_3(z),
    \end{equation}
    where $\Phi_1, \Phi_2, \Phi_3$ are appropriate double-well potentials.

    In particular, it shows that in the case the parameter $\varepsilon$ is larger than every scale of the microstructure, the surface density of the limiting functional is obtained by taking the weighted averages of the potentials.
\end{remark}

\begin{remark}\label{rem:H1}
In the proofs of the liminf and of the limsup inequalities, the unfolding operator forces us to evaluate the potential at $(g_n(y_1,y_2),y_2,p)$ instead of $(y_1,y_2,p)$, where
\[
g_n(y_1,y_2)\coloneqq \frac{\eta_n}{\delta_n}\left\lfloor \frac{\delta_n}{\eta_n}\,y_1\right\rfloor_{Q_2}+\frac{\eta_n}{\delta_n}\,y_2 .
\]
Since both $\big\{\frac{\delta_n}{\eta_n}y_1\big\}_{Q_2}$ and $y_2$ belong to $Q_2$, we have that
\begin{equation}\label{eq:gn_close}
|g_n(y_1,y_2)-y_1| = \frac{\eta_n}{\delta_n}\left| y_2-\left\{\frac{\delta_n}{\eta_n}\,y_1\right\}_{Q_2}\right| \leq \mathrm{diam}(Q_2)\,\frac{\eta_n}{\delta_n},
\end{equation}
for all $(y_1,y_2)\in Q_1\times Q_2$. Note that the third argument $p$ is not fixed, as it might depend on $n$.
Thus, what we need is a modulus of continuity of $W_i$ in the first variable which is uniform with respect to the third variable on bounded sets.
This is the reason why in \ref{itm:1_H1} we require the continuity of $W_i$ jointly in $(y_1,z)$.
Indeed, separate continuity in $y_1$ and in $z$ does not imply joint continuity, and it does not give a modulus of continuity in $y_1$ uniform with respect to $z$.
We stress that no uniformity with respect to $y_2$ is required.

\end{remark}

    \begin{remark}
        Assumption \ref{itm:1_H1} is satisfied for example in the following cases:
    \begin{itemize}
        \item $W(y_1,y_2,z) = f_1 (y_1) f_2 (y_2) \overline{W}(z)$, where we have $f_1$ continuous, $f_2$ bounded and $\overline{W}$ continuous.
        \item $W$ is the prototype potential of \eqref{eq:prototype_potential},
        where $\Phi_1,\Phi_2,\Phi_3$ are continuous and $I_1\subset Q_1$ is open with Lipschitz boundary. Indeed, it is enough to take $E_1\coloneqq I_1$, $E_2\coloneqq Q_1\setminus\overline{I_1}$, $W_1(y_1,y_2,z)\coloneqq \ca_{I_2}(y_2)\Phi_1(z)+\ca_{Q_2\setminus I_2}(y_2)\Phi_2(z)$, and $W_2(y_1,y_2,z)\coloneqq\Phi_3(z)$, which do not depend on $y_1$.
    \end{itemize}
    \end{remark}

\begin{remark}
    Our analysis is restricted to the case of two wells (see \ref{itm:1_H2}). For the case of multiple wells, the arguments that we develop are strong enough to be adapted with some minor modifications, by incorporating the techniques of \cite{baldo1990minimal}.
\end{remark}

\begin{remark}
We note that Assumption \ref{itm:1_H3} is only needed for the compactness.
In \cite[Theorem 1.6]{CriFonHagPop}, the compactness was obtained without the need of a spatially uniform lower bound on $W$.
Although it is possible to use a similar strategy as employed in \cite[Theorem 1.6]{CriFonHagPop} to get compactness without the need of a lower bound, we prefer to add this assumption in order to focus on the strategy to get the $\Gamma$-limit.
\end{remark}

\begin{remark}
    Assumption \ref{itm:1_H4} seems quite natural to us, as it implies that the potential energy does not blow up in a finite space.
    It will be used in the proofs to bound some integral terms containing $W$.
    A similar assumption appears in the work \cite{baldo1990minimal} by Baldo (see formula (1.2) in that paper).
\end{remark}

\begin{remark}
    Assumption \ref{itm:1_H5} is needed in order to get compactness. This is done by using the classical strategy developed in \cite{fonseca1989gradient} by Fonseca and Tartar (see also \cite{Leo_lecture}).

    Moreover, it is used in the proof of the liminf inequality to reduce to uniformly bounded sequences. In the case where a mass constraint is in force, Leoni showed in \cite{Leo} that, for the compactness, this growth condition can be relaxed. For the sake of simplicity, we do not pursue this here, and we assume \ref{itm:1_H5} also in Theorem \ref{thm:mass_constraint}.
\end{remark}

\begin{remark}\label{rem:uniform_diff}
    Assumption \ref{itm:1_H6} is only used in the proof of the mass constrained case (see Theorem \ref{thm:mass_constraint}). It is equivalent to $W(y_1,y_2,\cdot)$ being differentiable at $a$ with zero differential, uniformly with respect to $(y_1,y_2)$. Indeed, the definition of uniform differentiability gives
    \begin{equation}\label{eq:unif_diff}
    \lim_{z \to a} \sup_{y_1,y_2} \frac{| W(y_1,y_2,z) - W(y_1,y_2,a) - \partial_z W(y_1,y_2,a)(z-a)|}{|z-a|} = 0,
    \end{equation}
    from which we can define $\omega_a (r)$ for $r > 0$ as
    $$ \omega_a(r) = \sup_{0 < |z-a| \le r} \sup_{y_1,y_2} \frac{|W(y_1,y_2,z) - W(y_1,y_2,a) - \partial_z W(y_1,y_2,a)(z-a)|}{|z-a|}. $$
    This function is non-decreasing with respect to $r$, and thanks to \eqref{eq:unif_diff}, gives $\lim_{r \to 0^+} \omega_a (r) = 0$, and is therefore a valid modulus of continuity. Rearranging its definition gives
    $$| W(y_1,y_2,z) - W(y_1,y_2,a) - \partial_z W(y_1,y_2,a)(z-a)| \leq \omega_a (|z-a|) |z-a|,$$
    which is exactly our definition under the conditions that
    \[
    W \geq 0,\quad\quad\quad
    W(y_1,y_2,a) = 0,\quad\quad\quad
    \partial_z W(y_1,y_2,a) = 0,
    \]
    all satisfied in our case.
    In Remark \ref{rem:diff}, we provide a stronger assumption that is easier to check.
    \end{remark}

    \begin{remark}
        Assumption \ref{itm:1_H6} is satisfied for example in the following cases:
        \begin{itemize}
            \item there exist $C>0$ and $\beta > 1$ such that $W(y_1,y_2,z) \leq C |z-a|^\beta$, for $|z-a| \leq r_0$. In this case $\omega_a(s) = C s^{\beta-1}$.
            \item $W(y_1,y_2,z) = f_1(y_1) f_2 (y_2) \overline{W}(z)$, where we have $f_1, f_2$ bounded, and $\overline{W}$ differentiable in $z=a$.
            \item $W$ is the prototype potential of \eqref{eq:prototype_potential}, where $\Phi_1, \Phi_2, \Phi_3$ are of class $C^1$ in a neighborhood of $a$. Indeed, in this case $\omega_a(s) = \max_j \sup_{|z-a|\leq s} | \nabla \Phi_j (z) |$.
        \end{itemize}
    \end{remark}

We now introduce the functional that we will study.

\begin{definition}(\textbf{Integer and fractional decomposition})\label{def:fractional_decomposition} Let $G_1$ and $Q_1$ be as above. Then every $z \in \R^N$ can be decomposed with respect to $(G_1, Q_1)$ as
$$ z = \lfloor z \rfloor_{Q_1} + \{ z \}_{Q_1}, $$
where $\lfloor z \rfloor_{Q_1} \in G_1$ and $\{ z \}_{Q_1} \in Q_1$.
This definition can be stated analogously for points $y_1 \in Q_1$, decomposed with respect to $(G_2, Q_2)$.
\end{definition}

\begin{definition}\label{def:Fn}
    Let $\{\eta_n\}_n$, $\{\delta_n\}_n$, $\{\varepsilon_n\}_n$ be positive infinitesimal sequences.
    For each $n\in\N$, define $F_n \colon L^1 (\Omega; \R^M) \to [0,+\infty] $ as
    \begin{gather*}
        F_n (u) \coloneqq \begin{cases}
            \displaystyle\int_{\Omega} \Bigg[ \!\frac{1}{\varepsilon_n} W \Bigg( \!\bigg\{\! \frac{x}{\delta_n} \! \bigg\}_{Q_1}\!, \bigg\{\! \frac{\delta_n}{\eta_n}\!\bigg\{ \!\frac{x}{\delta_n} \!\bigg\}_{Q_1} \!\bigg\}_{Q_2}\!, u(x) \!\Bigg) + \varepsilon_n | \nabla u (x) |^2 \Bigg] \! \dd x \qquad &u \in W^{1,2}(\Omega; \R^M),\\
            +\infty \qquad &\text{otherwise}.
        \end{cases}
    \end{gather*}
\end{definition}

In the following, we will always require $\eta_n \ll \delta_n$, that is the microscales are separated.
Moreover, we will denote by $\mathrm{BV}(\Omega;\{a,b\})$ the set of functions of bounded variations taking only values in $\{a,b\}$.
The main results are the following.

\begin{theorem}(\textbf{Compactness})\label{thm:compactness}
Let $\{\eta_n\}_n$, $\{\delta_n\}_n$, $\{\varepsilon_n\}_n$ be positive infinitesimal sequences such that $\eta_n \ll \delta_n \ll \varepsilon_n$, that is
\begin{equation*}
    \lim_{n \to \infty} \cfrac{\eta_n}{\delta_n} = 0, \qquad \lim_{n \to \infty} \cfrac{\delta_n}{\varepsilon_n} = 0.
\end{equation*}
Assume that \ref{itm:1_H1}, \ref{itm:1_H2}, \ref{itm:1_H3}, and \ref{itm:1_H5} hold.
Let $\{u_n\}_n \subset L^1 (\Omega; \R^M) $ be such that 
$$ \sup_{n \in \N} F_n (u_n) < \infty. $$
Then, there exists a subsequence $\{u_{n_k}\}_k \subset W^{1,2}(\Omega; \R^M)$ and a function $u \in \mathrm{BV}(\Omega; \{ a,b\})$ such that $u_{n_k} \to u$ strongly in $L^1 (\Omega; \R^M)$.
\end{theorem}

We now define the limit functional.
\begin{definition}
    Let $u \in L^1 (\Omega; \R^M)$. Define $F_\infty \colon L^1 (\Omega; \R^M) \to [0,+\infty]$ as
    \begin{gather}\label{eq:gamma_limit}
        F_\infty (u) \coloneqq \begin{cases}
            \displaystyle\sigma^\mathrm{h} \, \mathrm{Per} \big( \{ u = a \}; \Omega \big) \qquad &u \in \mathrm{BV} ( \Omega; \{ a,b\}), \\
            +\infty \qquad &\text{otherwise},
        \end{cases}
    \end{gather}
    where $\sigma^\mathrm{h}$ is given by
    \begin{equation}\label{eq:sigma_gamma_limit}
        \sigma^\mathrm{h} \coloneqq \inf \left\{ \int_{-1}^1 2 \sqrt{W^\mathrm{h} ( \gamma(t))} \, | \gamma' (t)| \, \dd t : \gamma \in \text{Lip}([-1,1]; \R^M ), \gamma(-1)=a, \gamma(1)=b \right\}, 
    \end{equation}
    and $W^\mathrm{h}$ is defined as
    \begin{equation}\label{eq:homogenized}
        W^\mathrm{h} (z) \coloneqq \int_{Q_1} \int_{Q_2} W(y_1,y_2,z) \, \dd y_1 \dd y_2,
    \end{equation}
    for all $z\in\R^M$.
\end{definition}

\begin{theorem}(\;\textbf{$\Gamma$-convergence})\label{thm:gamma_convergence} Let $\{\eta_n\}_n$, $\{\delta_n\}_n$, $\{\varepsilon_n\}_n$ be positive infinitesimal sequences such that $\eta_n \ll \delta_n \ll \varepsilon_n$, that is
\begin{equation*}
    \lim_{n \to \infty} \cfrac{\eta_n}{\delta_n} = 0, \qquad \lim_{n \to \infty} \cfrac{\delta_n}{\varepsilon_n} = 0.
\end{equation*}
Assume that  \ref{itm:1_H1}-\ref{itm:1_H5} hold.
Then, as $n \to \infty$,  $F_n$ $\Gamma$-converges to $F_\infty$ with respect to the strong $L^1 (\Omega; \R^M)$ convergence.
\end{theorem}

\begin{figure}
    \centering
    \input{fig_first_regime}
    \caption{The regime considered in this paper: $\eta \ll \delta \ll \varepsilon$}
    \label{fig:first_regime}
\end{figure}

The proofs for compactness and $\Gamma$-convergence are robust enough to be applied (with minor modifications) to the case of a mass-constrained functional.
\begin{definition}(\textbf{Mass-constrained case})
    Let $m \in \R^M$ be the target mass. We define the mass-constrained functionals $\widehat{F}_n \colon  L^1 (\Omega; \R^M) \to [0,+\infty]$ as
    \begin{gather*}
        \widehat{F}_n (u) \coloneqq \begin{cases} \displaystyle
            \displaystyle F_n (u) \qquad &u \in W^{1,2}(\Omega;\R^M), \ \displaystyle\int_\Omega u \dd x = m, \\
            + \infty \qquad &\text{otherwise},
        \end{cases}
    \end{gather*}
    and $\widehat{F}_\infty \colon L^1 (\Omega; \R^M) \to [0,+\infty]$ as
    \begin{gather*}
        \widehat{F}_\infty (u) \coloneqq \begin{cases}
            \displaystyle F_\infty (u) \qquad &u \in \mathrm{BV}(\Omega;\{ a,b \}), \ \displaystyle\int_\Omega u \dd x = m, \\
            + \infty \qquad &\text{otherwise}.
        \end{cases}
    \end{gather*}
\end{definition}

\begin{theorem}\label{thm:mass_constraint}
    Let $m \in \R^M$ be the target mass, and let $\{\eta_n\}_n$, $\{\delta_n\}_n$, $\{\varepsilon_n\}_n$ be positive infinitesimal sequences such that $\eta_n \ll \delta_n \ll \varepsilon_n$.
Then, the following hold:
\begin{enumerate}\setlength\itemsep{0.2cm}
    \item Let $\{u_n\}_n \subset L^1 (\Omega; \R^M) $ be a sequence of functions such that 
    $$ \sup_{n \in \N} \widehat{F}_n (u_n) < \infty. $$
    Assume that \ref{itm:1_H1}, \ref{itm:1_H2}, \ref{itm:1_H3}, and \ref{itm:1_H5} hold.
    Then, there exists a subsequence $\{u_{n_k}\}_k \subset W^{1,2}(\Omega; \R^M)$ and a function $u \in \mathrm{BV}(\Omega; \{ a,b\})$ such that
    $$ u_{n_k} \to u \quad \text{strongly in } L^1 (\Omega; \R^M). $$
    \item Assume that  \ref{itm:1_H1}-\ref{itm:1_H6} hold.
    Then, $\{\widehat{F}_n\}_n$ $\Gamma$-converges with respect to the strong $L^1 (\Omega; \R^M)$ convergence to $\widehat{F}_\infty$, as $n \to \infty$.
\end{enumerate}
\end{theorem}

\begin{remark}
    In the case $N\geq 1$ and $M=1$, an argument based on translations of the optimal profile allows to obtain the required mass-constrained recovery sequence, under no additional regularity assumptions (see \cite{Leo_lecture}). On the other hand, in the case $M > 1$ such a strategy does not work. This is why we took inspiration from an argument by Murray (see \cite{murray2016slow}) and Baldo (see \cite{baldo1990minimal}). Under Assumption \ref{itm:1_H6}, the construction works for all $M>1$.
\end{remark}

\begin{remark}
In case the potential $W$ is assumed to be globally $C^2$ in the last variable, a simpler proof gives the above result for all $N,M\geq 1$ (see \cite{Ishige}).
\end{remark}


\subsection{Outline of the strategy}

In this section, we explain the main novel ideas and the challenges of the proofs of the main result, Theorem \ref{thm:gamma_convergence}.

For the liminf inequality, the main challenge is to first take the limit as $\eta_n,\delta_n\to0$ while keeping $\varepsilon_n$ fixed, and then sending $\varepsilon_n\to0$.
We first give a brief overview of the strategy used in \cite{CriFonGan_fixed_Sup} for the case of one scale of microstructure, in order to highlight the differences with the novel and simpler argument that we develop here.
The idea is the following: given a sequence $\{u_n\}_{n\in\N}\subset W^{1,2}(\Omega;\R^M)$ such that $u_n\to u$ for some $u\in \mathrm{BV}(\Omega;\{a,b\})$, we focus on the potential term
\[
\int_\Omega W\bigg(\bigg\{\frac{x}{\delta_n}\bigg\}_{Q_1}, u_n(x)\bigg) \dd x.
\]
We would like to replace the integrand
\[
W\bigg(\bigg\{\frac{x}{\delta_n}\bigg\}_{Q_1}, u_n(x)\bigg) \quad\quad\quad \text{ with }\quad\quad\quad
W^\mathrm{h}\left( u_n(x) \right).
\]
This, in general, requires the \emph{technical} strong assumption $\delta\ll\varepsilon^{3/2}$ to use a Poincar\'{e} inequality to perform the pointwise substitution (see \cite{Hag}, and also \cite{AnsBraChi2}) since the above integral is multiplied by $\varepsilon_n^{-1}$.
Since we are interested in a liminf inequality, we reason as follows. Using the unfolding operator $\mathcal{U}_1$ at scale $\delta_n$, we can write (up to asymptotically negligible terms)
\[
\int_\Omega W\bigg(\bigg\{\frac{x}{\delta_n}\bigg\}_{Q_1}, u_n(x)\bigg) \dd x
 = \int_\Omega \int_{Q_1} W(y, \mathcal{U}_1 u_n(x,y)) \dd y \dd x.
\]
If in the integrand on the right-hand side we had $u_n(x)$ in place of $\mathcal{U}_1 u_n(x,y)$, we could conclude, since that term would exactly be the homogenized potential $W^\mathrm{h} ( u_n(x))$.
Having this in mind, we write
\[
\int_\Omega \int_{Q_1} W(y, \mathcal{U}_1 u_n(x,y)) \dd y \dd x
    = \int_\Omega \int_{Q_1} W(y, u_n(x) + (\mathcal{U}_1 u_n(x,y) -u_n(x))) \dd y \dd x,
\]
and we treat the term $\mathcal{U}_1 u_n(x,y) - u _n(x)$ as a perturbation that we expect to vanish as $n\to\infty$. In particular, for a fixed $\xi>0$, we can consider the space of such admissible perturbations $\mathcal{A}_\xi$ that will have to satisfy some technical conditions that we do not specify in here (see \cite[Definition 3.3]{CriFonGan_fixed_Sup}), but with the fundamental property that $\mathcal{A}_\xi$ contains only the zero function when $\xi=0$. Thus, we get
\begin{equation}\label{eq:W_xi_outline}
\int_\Omega \int_{Q_1} W(y, u_n(x) + (\mathcal{U}_1 u_n(x,y) -u_n(x))) \dd y \dd x \geq \int_\Omega W^\xi(u_n(x)) \dd x,
\end{equation}
where
\[
W^\xi(p) \coloneqq \inf\left\{ \int_{Q_1} W(y, p + \psi(y)) \dd y : \psi\in\mathcal{A}_\xi  \right\}.
\]
Therefore, we got rid of the scale $\delta_n$, at the cost of introducing a slightly different potential $W^\xi$ in place of $W^\mathrm{h}$.
This will need to be fixed at the end of the proof. Indeed, using \eqref{eq:W_xi_outline}, the standard argument for the liminf of the Modica-Mortola functional yields that
\[
\liminf_{n\to\infty} F_n(u_n) \geq \sigma^\xi\,\, \mathrm{Per}(\{u=a\};\Omega),
\]
where
\[
\sigma^\xi \coloneqq \inf \left\{ \int_{-1}^1 2 \sqrt{W^\xi ( \gamma(t))} \, | \gamma' (t)| \, \dd t : \gamma \in \text{Lip}([-1,1]; \R^M ), \gamma(-1)=a, \gamma(1)=b \right\},
\]
Therefore, in order to conclude, we need to prove that
\[
\lim_{\xi\to0} \sigma^\xi = \sigma^\mathrm{h}.
\]
This is done by a study of the geodesic problem defining those two quantities, using the fact that $W^\xi$ converges to $W^\mathrm{h}$ locally uniformly.\\

For multiple scales, one would expect a similar strategy to work.
One of the main contributions of the present paper is to show that this is the case, and that the resulting argument is actually simpler than the one for a single scale. If for a single scale the periodic unfolding is used, for two scales a \emph{double periodic unfolding} $\mathcal{U}_2$ is needed. This mathematical tool has been recently developed by Damlamian, Griso, and Cioranescu in \cite{damlamian2022periodic} (see Section \ref{subsec:keyestimates} for more details) and allows us to write
\begin{align*}
    \int_\Omega W &\Bigg( \!\bigg\{\! \frac{x}{\delta_n} \! \bigg\}_{Q_1}\!, \bigg\{\! \frac{\delta_n}{\eta_n}\!\bigg\{ \!\frac{x}{\delta_n} \!\bigg\}_{Q_1} \!\bigg\}_{Q_2}\!, u_n(x) \!\Bigg) \dd x \\
    &\qquad\qquad\qquad\approx \int_{\Omega} \int_{Q_1} \int_{Q_2}\! W \!\bigg( \frac{\eta_n}{\delta_n} \bigg\lfloor \frac{\delta_n}{\eta_n}y_1 \bigg\rfloor_{Q_2}\! + \frac{\eta_n}{\delta_n}y_2, y_2, \mathcal{U}_2 u_n(x,y_1,y_2) \! \bigg) \!\dd y_2 \dd y_1 \dd x.
\end{align*}
We then write $\mathcal{U}_2 u_n (x,y_1,y_2) = u_n(x) + \psi_n(x,y_1,y_2)$, and treat $\psi_n$ as a perturbation.

Two difficulties arise. First, the perturbation $\psi_n$ is small only after integration with respect to $x$, while the replacement of the integrand has to be performed for each fixed $x$. We overcome this by proving a fiber-wise estimate for the double unfolding (see Lemma \ref{lemma:L2_liminf_estimates}), namely
$$ \int_{\widehat \Omega_{\delta_n}} \| \psi_n (x, \cdot, \cdot)\|^2_{L^2} \dd x \leq C \frac{\delta_n^2 + \eta_n^2}{\varepsilon_n}. $$
This shows that the set of points $x$ where $\| \psi_n (x, \cdot,\cdot)\|_{L^2}$ is not small has measure $o(\varepsilon_n)$, and thus it is negligible even after being divided by $\varepsilon_n$.

Second, after unfolding, the first argument of $W$ is not $y_1$, but a point at distance of order $\eta_n/\delta_n$ from it, and the substitution has to be performed uniformly with respect to the value of the unfolded function. This is where the continuity of the $W_i$'s jointly in $(y_1,z)$ is used (see Lemma \ref{lemma:uniform_replacement}).

Combining these two facts, we prove a uniform lower bound for the cell problem (see Lemma \ref{lemma:liminf_lemma_amazing}): for every $t > 0$ there exists $\xi > 0$ such that the unfolded integrand is larger than $W^\mathrm{h}(u_n(x)) - t$, wherever $\| \psi_n (x, \cdot, \cdot)\|_{L^2} \leq \xi$ and $n$ is large. This allows us to bound the potential term from below with $\varepsilon_n^{-1} \int_\Omega V_t (u_n) \dd x$, where $V_t = (W^\mathrm{h}-t)^+$ does not depend on $n$. The classical Modica-Mortola argument then gives
$$ \liminf_{n \to \infty} F_n (u_n) \geq \sigma_t \mathrm{Per}(\{u=a\};\Omega), $$
where $\sigma_t$ is the surface tension associated with $V_t$, and we conclude by showing that $\sigma_t \to \sigma^\mathrm{h}$ as $t \to 0^+$ (see Lemma \ref{lemma:stability_sigma}).

Our strategy has several advantages over the one of \cite[Theorem 4.1]{CriFonGan_fixed_Sup}: first of all, we do not need to introduce an auxiliary cell problem with classes of admissible perturbations; second,  we do not need a deeper study of the geodesic problem for a potential whose zero set consists of compact sets, but only to understand its asymptotics (see Lemma \ref{lemma:stability_sigma}). 
Moreover, the strategy that we develop here is significantly easier both in terms of argument and in terms of length.
Finally, our argument extends, with only notational changes, to any finite number of separated microscales.\\

The construction of the limsup inequality follows a standard approximation argument, where we reduce ourselves to studying the case of a single flat interface. Since the quantity that we want to approximate is $\sigma^\mathrm{h}$, we take an almost infimizing curve, admissible for the problem defining $\sigma^\mathrm{h}$, and, in a tubular neighborhood of size $\varepsilon_n$, we rescale it in the normal direction to the interface. The novelty for the limsup lies in performing delicate estimates to ensure that, for this particular sequence $\{u_n\}_n$, it holds
\[
\limsup_{n\to\infty} \int_\Omega \left[ W \Bigg( \!\bigg\{\! \frac{x}{\delta_n} \! \bigg\}_{Q_1}\!, \bigg\{\! \frac{\delta_n}{\eta_n}\!\bigg\{ \!\frac{x}{\delta_n} \!\bigg\}_{Q_1} \!\bigg\}_{Q_2}\!, u(x) \!\Bigg) - W^\mathrm{h}(u_n(x)) \right] \dd x \leq 0,
\]
which is based on a delicate continuity argument for $W$ and for the function $u_n$, different from the one in \cite{CriFonGan_fixed_Sup}, which relies on a partition of the domain for each periodicity cell.


\section{Technical results}

From this section on, unless explicitly stated, we denote by $C$ a positive universal constant positive that may vary from line to line.


\subsection{Key estimates}\label{subsec:keyestimates}

We recall the classical notion of single scale unfolding operator defined in \cite{cioranescu2008periodic} and recall the new notion of double scale unfolding operator, as defined in \cite{damlamian2022periodic}.



\begin{definition}(\textbf{First unfolding operator})\label{def:first_unfolding}
    Let us define:
    \begin{itemize}\setlength\itemsep{0.2cm}
        \item[$\circ$] $ \displaystyle \Xi_1 \coloneqq \left\{ \xi \in G_1 : \delta (\xi + Q_1) \subset \Omega \right\} $,
        \item[$\circ$] $ \displaystyle \widehat{\Omega}_{\delta} \coloneqq \bigcup_{\xi \in \Xi_1} \delta \left(\xi + \overline{Q}_1 \right)$,
        \item[$\circ$] $ \displaystyle \Lambda_{\delta} \coloneqq \Omega \setminus \widehat{\Omega}_{\delta}$.
    \end{itemize}
    The \emph{first unfolding operator} $\mathcal{U}_1 \colon L^2 (\Omega; \R^M) \to L^2 (\Omega; L^2 (Q_1; \R^M))$ is defined as
    \begin{gather*}
        \mathcal{U}_1 \phi (x,y_1) \coloneqq \begin{cases}
            \displaystyle\phi \left( \delta \left\lfloor \frac{x}{\delta} \right\rfloor_{Q_1} + \delta y_1 \right) \quad &\text{a.e. in } \widehat{\Omega}_{\delta} \times Q_1, \\
            a \qquad &\text{otherwise}.
        \end{cases}
    \end{gather*}
\end{definition}

\begin{figure}
    \centering
    \input{fig_unfolding}
    \caption{Left: Integer and fractional decomposition in $\R^N$.
    Right: Integer and fractional decomposition in $\Omega$.}
    \label{fig:unfolding_operator}
\end{figure}

\begin{definition}(\textbf{Second unfolding operator})\label{def:second_unfolding}
    Let us define:
    \begin{itemize}\setlength\itemsep{0.2cm}
        \item[$\circ$] $ \displaystyle \Xi_2 \coloneqq \left\{ \xi \in G_2 : \frac{\eta}{\delta} (\xi + Q_2) \subset Q_1 \right\} $,
        \item[$\circ$] $ \displaystyle \widehat{Q}_{1, \eta} \coloneqq \bigcup_{\xi \in \Xi_2} \frac{\eta}{\delta} \left( \xi + \overline{Q}_2 \right)$,
        \item[$\circ$] $ \displaystyle \Lambda_{1,\eta} \coloneqq Q_1 \setminus \widehat{Q}_{1,\eta}$.
    \end{itemize}
    The \emph{second unfolding operator} $\mathcal{U}_2 \colon L^2 (\Omega; \R^M) \to L^2(\Omega; L^2 (Q_1; L^2( Q_2; \R^M)))$ is defined as
    \begin{gather*}
        \mathcal{U}_2 \phi (x,y_1,y_2) \coloneqq \begin{cases}
            \displaystyle\phi \left( \delta \left\lfloor \frac{x}{\delta} \right\rfloor_{Q_1} + \eta \left\lfloor \frac{\delta y_1}{\eta} \right\rfloor_{Q_2} + \eta y_2 \right) \; &\text{a.e. in } \widehat{\Omega}_{\delta} \times \widehat{Q}_{1, \eta} \times Q_2, \\
            a \qquad &\text{otherwise}.
        \end{cases}
    \end{gather*}
\end{definition}

These definitions are slightly different from the classical ones, as in this case the operators are equal to one of the wells on the boundary sets. This is just a choice, but defining it as any other value would work, since all of our computations are made far away from the boundary.

\begin{remark}\label{remark:alternative_potential}
    In some formulations of the unfolding operators, there is the need for a sequence of vectors $\iota_{2,\eta} \in \overline{Q}_2$ called adjustments. This depends on the shape of the potential, and for our choice it was not needed. If we assumed a simpler potential like $W \left( \frac{x}{\delta}, \frac{x}{\eta}, u \right)$, $G_1$-periodic in $y_1$ and $G_2$-periodic in $y_2$, then the adjustments would be needed to account for the mismatch between the two nested microscales, in which case we would need
    \begin{equation*}
        \iota_{2,\eta} \coloneqq \left\{ \frac{\delta}{\eta} \left\lfloor \frac{x}{\delta} \right\rfloor_{Q_1} \right\}_{Q_2} \in \overline{Q_2},
    \end{equation*}
    and the definition of the second unfolding operator would change to
    \begin{gather*}
        \mathcal{U}_2 \phi (x,y_1,y_2) \coloneqq \begin{cases}
            \displaystyle\phi \left( \delta \left\lfloor \frac{x}{\delta} \right\rfloor_{Q_1} + \eta \left\lfloor \frac{\delta y_1}{\eta} \right\rfloor_{Q_2} - \eta \, \iota_{2,\eta} + \eta y_2 \right) \; &\text{a.e. in } \widehat{\Omega}_{\delta} \times \widehat{Q}_{1, \eta} \times Q_2, \\
            a \qquad &\text{otherwise}.
        \end{cases}
    \end{gather*}
    The results of this paper still hold for this different potential, by adopting this definition of unfolding operator, and slightly modifying the proofs to account for the adjustments. For further details, we refer to \cite{damlamian2022periodic}.
\end{remark}

\begin{lemma}\label{lemma:needed}
    Let $W \colon Q_1 \times Q_2 \times \R^M \to [0,+\infty)$ satisfying \ref{itm:1_H1}, and let $u \in L^2(\Omega; \R^M)$.
    Then, we have
    \begin{equation}
        \int_{\Omega} W \Bigg( \bigg\{ \frac{x}{\delta}  \bigg\}_{Q_1}, \bigg\{ \frac{\delta}{\eta}\bigg\{ \frac{x}{\delta} \bigg\}_{Q_1} \bigg\}_{Q_2}, u \Bigg) \dd x \geq \int_{\widehat\Omega_\delta} \int_{\widehat Q_{1,\eta}} \int_{Q_2} W \left( g(y_1,y_2), y_2, \mathcal{U}_2 u  \right) \dd y_2 \dd y_1 \dd x,
    \end{equation}
    where $g(y_1,y_2)$ is defined as
    \[
        g(y_1,y_2)\coloneqq \frac{\eta}{\delta}\left\lfloor \frac{\delta}{\eta}\,y_1\right\rfloor_{Q_2}+\frac{\eta}{\delta}\,y_2.
    \]
\end{lemma} 

\begin{proof}
    Using the definition of the first unfolding operator $\mathcal{U}_1$ we have
    \begin{align*}
        \int_{\Omega} W \Bigg( \!\bigg\{\! \frac{x}{\delta} \! \bigg\}_{Q_1}\!, \bigg\{\! \frac{\delta}{\eta}\!\bigg\{ \!\frac{x}{\delta} \!\bigg\}_{Q_1} \!\bigg\}_{Q_2}\!, u \!\Bigg) \dd x &\geq \int_{\widehat{\Omega}_\delta} W \Bigg( \!\bigg\{\! \frac{x}{\delta} \! \bigg\}_{Q_1}\!, \bigg\{\! \frac{\delta}{\eta}\!\bigg\{ \!\frac{x}{\delta} \!\bigg\}_{Q_1} \!\bigg\}_{Q_2}\!, u \!\Bigg) \dd x \\
        &= \sum_{\xi_1 \in \Xi_1} \int_{\delta \xi_1 + \delta Q_1} W \Bigg( \!\bigg\{\! \frac{x}{\delta} \! \bigg\}_{Q_1}\!, \bigg\{\! \frac{\delta}{\eta}\!\bigg\{ \!\frac{x}{\delta} \!\bigg\}_{Q_1} \!\bigg\}_{Q_2}\!, u \!\Bigg) \dd x \\
        &= \sum_{\xi_1 \in \Xi_1} \delta^N \int_{Q_1} W \left( y_1, \bigg\{ \frac{\delta}{\eta} y_1 \bigg\}_{Q_2}, u \left( \delta \xi_1 + \delta y_1 \right) \right) \dd y_1 \\
        &= \sum_{\xi_1 \in \Xi_1} \int_{\delta \xi_1 + \delta Q_1} \int_{Q_1} W \left( y_1, \bigg\{ \frac{\delta}{\eta} y_1 \bigg\}_{Q_2}, u \left( \delta \xi_1 + \delta y_1 \right) \right) \dd y_1 \dd x,
    \end{align*}
    where in the second to last equality we used the change of coordinates $x = \delta \xi_1 + \delta y_1$. We now repeat this estimate using the second partial unfolding operator:
    \begin{align*}
        &\sum_{\xi_1 \in \Xi_1} \int_{\delta \xi_1 + \delta Q_1} \int_{Q_1} W \left( y_1, \bigg\{ \frac{\delta}{\eta} y_1 \bigg\}_{Q_2}, u \left( \delta \xi_1 + \delta y_1 \right) \right) \dd y_1 \dd x \\
        &\geq \sum_{\xi_1 \in \Xi_1} \int_{\delta \xi_1 + \delta Q_1} \int_{\widehat{Q}_{1,\eta}} W \left( y_1, \bigg\{ \frac{\delta}{\eta} y_1 \bigg\}_{Q_2}, u \left( \delta \xi_1 + \delta y_1 \right) \right) \dd y_1 \dd x \\
        &= \sum_{\xi_1 \in \Xi_1} \int_{\delta \xi_1 + \delta Q_1} \sum_{\xi_2 \in \Xi_2} \int_{\frac{\eta}{\delta} \xi_2 + \frac{\eta}{\delta} Q_2} W \left( y_1, \bigg\{ \frac{\delta}{\eta} y_1 \bigg\}_{Q_2}, u \left( \delta \xi_1 + \delta y_1 \right) \right) \dd y_1 \dd x \\
        &= \sum_{\xi_1 \in \Xi_1} \int_{\delta \xi_1 + \delta Q_1} \sum_{\xi_2 \in \Xi_2} \left( \frac{\eta}{\delta} \right)^N \int_{Q_2} W \left( \frac{\eta}{\delta} \xi_2 + \frac{\eta}{\delta} y_2, y_2, u \left( \delta \xi_1 + \eta \xi_2 + \eta y_2 \right) \right) \dd y_2 \dd x \\
        &= \sum_{\xi_1 \in \Xi_1} \sum_{\xi_2 \in \Xi_2} \int_{\delta \xi_1 + \delta Q_1} \int_{\frac{\eta}{\delta} \xi_2 + \frac{\eta}{\delta} Q_2} \int_{Q_2} W \left( \frac{\eta}{\delta} \xi_2 + \frac{\eta}{\delta} y_2, y_2, u \left( \delta \xi_1 + \eta \xi_2 + \eta y_2 \right) \right) \dd y_2 \dd y_1 \dd x,
    \end{align*}
    where in the second to last equality we used the change of coordinates $y_1 = \frac{\eta}{\delta} \xi_2 + \frac{\eta}{\delta} y_2$. We substitute now the explicit formulas for $\xi_1$ and $\xi_2$ to obtain
    \begin{align*}
        &\sum_{\xi_1 \in \Xi_1} \sum_{\xi_2 \in \Xi_2} \int_{\delta \xi_1 + \delta Q_1} \int_{\frac{\eta}{\delta} \xi_2 + \frac{\eta}{\delta} Q_2} \int_{Q_2} W \left( \frac{\eta}{\delta} \xi_2 + \frac{\eta}{\delta} y_2, y_2, u \left( \delta \xi_1 + \eta \xi_2 + \eta y_2 \right) \right) \dd y_2 \dd y_1 \dd x \\
        &= \int_{\widehat{\Omega}_\delta} \int_{\widehat{Q}_{1,\eta}} \int_{Q_2} W \left( \frac{\eta}{\delta} \left\lfloor \frac{\delta}{\eta} y_1 \right\rfloor_{Q_2} \! + \frac{\eta}{\delta} y_2, y_2, u \left( \delta \left\lfloor \frac{x}{\delta} \right\rfloor_{Q_1} + \eta \left\lfloor \frac{\delta y_1}{\eta} \right\rfloor_{Q_2} + \eta y_2 \right) \right) \dd y_2 \dd y_1 \dd x,
    \end{align*}
    which gives the exact definition of the second unfolding operator on $\widehat{\Omega}_{\delta} \times \widehat{Q}_{1, \eta} \times Q_2$, and we conclude.
\end{proof}

The following lemma takes care of the dependence on $n$ of the first argument in the potential in the right-hand side of Lemma \ref{lemma:needed}, by making sure that we can substitute $g_n (y_1,y_2)$ with $y_1$ up to a negligible error (see Remark \ref{rem:H1}).

\begin{lemma}\label{lemma:uniform_replacement}
Assume \ref{itm:1_H1} and \ref{itm:1_H4}, and let $S>0$. Then there exists a non-decreasing function $\omega_S \colon [0,+\infty) \to [0,+\infty)$ with $\lim_{r\to0^+} \omega_S(r) = 0$, and a constant $C>0$ depending only on $Q_1$ and $Q_2$, such that the following holds. \\
Let $n\in\N$, and let $p,q\colon Q_1\times Q_2\to\R^M$ be measurable functions with $|p|\leq S$ and $|q|\le S$. Then
\begin{equation}\label{eq:uniform_replacement}
\int_{\widehat{Q}_{1,\eta_n}} \int_{Q_2}\big|W(g_n(y_1,y_2),y_2,p(y_1,y_2))-W(y_1,y_2,q(y_1,y_2))\big|\dd y_2\dd y_1
\leq \omega_S\left(\frac{\eta_n}{\delta_n}+\|p-q\|_{L^\infty}\right).
\end{equation}
\end{lemma}

\begin{figure}
    \centering
    \input{fig_substitution}
    \caption{The cell $Q_1$ with the partition $\{E_1,E_2\}$, subdivided into the cells $\frac{\eta}{\delta}(\xi_2+Q_2)$.
    Shaded: the layer $\Lambda_{1,\eta}$ of cells not contained in $Q_1$, and the strip of width of order $\eta/\delta$ around $\partial E_1$ (the part of $T_2$ along $\partial Q_1$ is not shaded).
    Right: since $g_n(y_1,y_2)$ lies in the same cell as $y_1$, it lies in the same set $E_i$ unless $y_1$ is in the strip.
    Both bad sets have measure $O(\eta/\delta)$.}
    \label{fig:substitution}
\end{figure}

\begin{proof} For $i\in\{1,\dots,K\}$, for $y_2\in Q_2\setminus\mathcal{N}$ and for $r\geq0$, we define
\[
\omega_i(y_2;r)\coloneqq \sup\Big\{ |W_i(y_1,y_2,z)-W_i(y_1',y_2,z')| : y_1,y_1'\in\overline{E_i},\ z,z'\in\overline{B_S},\ |y_1-y_1'|+|z-z'|\leq r \Big\},
\]
and set $\omega_i(y_2;r)\coloneqq 2C_S$ for $y_2\in\mathcal{N}$, where $C_S$ is the constant given by \ref{itm:1_H4}.

For $y_2\in Q_2\setminus\mathcal{N}$, the map $(y_1,z)\mapsto W_i(y_1,y_2,z)$ is jointly continuous and therefore uniformly continuous in the compact set $\overline{E_i}\times\overline{B_S}$. Therefore $\omega_i(y_2;r) \to 0$ as $r \to 0^+$.

Moreover, by \ref{itm:1_H4} we have $\omega_i(y_2;r)\leq 2C_S$ for $\ln$-a.e. $y_2\in Q_2$, and the function $y_2\mapsto\omega_i(y_2;r)$ is measurable by definition.

Since the sets $E_i$ are bounded with Lipschitz boundary, there exists $C_E>0$ such that
\begin{equation}\label{eq:minkowski}
\ln\big(\{ y_1\in E_i : \mathrm{dist}(y_1,\R^N\setminus E_i)\leq s \}\big)\leq C_E\, s \qquad\text{for all } s>0,\ i=1,\dots,K.
\end{equation}

Let $C_2\coloneqq\max\{1,\mathrm{diam}(Q_2)\}$, and define
\[
\omega_S(r)\coloneqq \sum_{i=1}^K\int_{Q_2}\omega_i(y_2;C_2r)\dd y_2 + 2C_S K C_E C_2 r.
\]
This function is non-decreasing, and, by the Dominated Convergence Theorem, $\omega_S(r)\to0$ as $r\to0^+$.

We claim that
\begin{equation}\label{eq:est_cont}
\int_{\widehat Q_{1,\eta}}\int_{Q_2}\big|W(g_n(y_1,y_2),y_2,p)-W(y_1,y_2,q)\big|\leq \omega_S\left(\frac{\eta_n}{\delta_n}+\|p-q\|_\infty \right).
\end{equation}
First, fix $i\in\{1,\dots,K\}$, and let
\[
T_i^n\coloneqq\left\{ y_1\in E_i : \mathrm{dist}(y_1,\R^N\setminus E_i)\leq C_2\frac{\eta_n}{\delta_n}\right\}.
\]
Using \eqref{eq:minkowski}, we get
\begin{equation}\label{eq:est_T_i}
|T_i^n|\le C_E C_2 \frac{\eta_n}{\delta_n}.
\end{equation}
Take $y_1\in (E_i\setminus T_i^n)\cap \widehat{Q}_{1,\eta_n}$ and $y_2\in Q_2\setminus\mathcal{N}$.
Using \eqref{eq:gn_close}, we get 
$$|g_n(y_1,y_2)-y_1|\le C_2\frac{\eta_n}{\delta_n} <\mathrm{dist}(y_1,\R^N\setminus E_i)$$
and hence $g_n(y_1,y_2)\in E_i$ (see Figure \ref{fig:substitution}).
Therefore 
$$W(g_n(y_1,y_2),y_2,p)=W_i(g_n(y_1,y_2),y_2,p), \qquad W(y_1,y_2,q)=W_i(y_1,y_2,q),$$
and thus
\[
\big|W(g_n(y_1,y_2),y_2,p)-W(y_1,y_2,q)\big|\leq \omega_i\left(y_2;C_2\frac{\eta_n}{\delta_n} + |p-q|\right)\le \omega_i\left(y_2;C_2 \left(\frac{\eta_n}{\delta_n}+\|p-q\|_{L^\infty}\right)\right).
\]
On the other hand, for $y_1\in T_i^n\cap \widehat{Q}_{1,\eta_n}$, we simply use that both terms are bounded by $C_S$.

Integrating \eqref{eq:est_cont}, and using \eqref{eq:est_T_i}, we obtain
\begin{align*}
    \int_{\widehat{Q}_{1,\eta_n}}\int_{Q_2}&\big|W(g_n,y_2,p)-W(y_1,y_2,q)\big|\dd y_2\dd y_1\\
    &\leq \sum_{i=1}^K\int_{Q_2}\omega_i\left(y_2;C_2\left(\frac{\eta_n}{\delta_n}+\|p-q\|_{L^\infty}\right)\right)\dd y_2+2C_S K C_E C_2 \frac{\eta_n}{\delta_n} \\
    &\leq \omega_S\left(\frac{\eta_n}{\delta_n}+\|p-q\|_\infty \right),
\end{align*}
and we conclude.
\end{proof}

We now prove some properties of the homogenized potential.

\begin{lemma}\label{lemma:continuity_homogenized}
    Assume \ref{itm:1_H1}, \ref{itm:1_H2} and \ref{itm:1_H4}. Then $W^\mathrm{h}$, defined in \eqref{eq:homogenized}, is continuous and $W^\mathrm{h}(z) = 0$ if and only if $z \in \{a,b\}$. If \ref{itm:1_H5} holds as well, then $W^\mathrm{h}(z) \geq | z| / R$ for all $|z|\geq R$. 
\end{lemma}

\begin{proof}
    Let $z_k \to z$, and let $S > 0$ be such that $|z_k | \leq S $ for all $k$. Thanks to \ref{itm:1_H1}, for a.e. $y_1,y_2$ we have that $W(y_1, y_2, \cdot)$ is continuous, and therefore
    $$ W(y_1, y_2, z_k) \to W(y_1, y_2, z)\qquad \text{for a.e. } y_1, y_2. $$
    Since by \ref{itm:1_H4} we have $W(y_1,y_2,z_k) \leq C_S$, we can apply the Dominated Convergence Theorem to get
    $$ W^\mathrm{h}(z_k) = \int_{Q_1} \int_{Q_2} W(y_1,y_2,z_k) \dd y_2 \, \dd y_1 \to \int_{Q_1} \int_{Q_2} W(y_1,y_2,z) \dd y_2 \, \dd y_1 = W^\mathrm{h} (z), $$
    hence $W^\mathrm{h}$ is continuous. Now, if $W^\mathrm{h}(z) = 0$, by positivity of $W$ this implies $W(y_1,y_2,z) = 0$ for a.e. $y_1,y_2$, and \ref{itm:1_H2} gives $z \in \{a,b\}$; the converse follows directly from \ref{itm:1_H2}.

    The last statement follows by integrating \ref{itm:1_H5} and using $|Q_1|, |Q_2| = 1$.
\end{proof}

We now prove the following lemma, which lets us lower bound the potential energy on the periodicity cells with the homogenized potential, up to a small error.

\begin{lemma}\label{lemma:liminf_lemma_amazing}
    Assume \ref{itm:1_H1} and \ref{itm:1_H4}, and let $S > 0$. Then, for every $t > 0$, there exist $\xi > 0$ and $n_0 \in \N$ such that
    \begin{equation}\label{eq:liminf_lemma_amazing}
        \int_{\widehat Q_{1,\eta_n}} \int_{Q_2} W (g_n(y_1, y_2), y_2, z + \psi(y_1,y_2)) \dd y_2 \, \dd y_1 \geq W^\mathrm{h}(z) - t,
    \end{equation}
    for all measurable $\psi \colon \widehat Q_{1, \eta_n} \times Q_2 \to \R^M$ with $| \psi | \leq S$ and
    \[
    \| \psi \|_{L^2(\widehat Q_{1, \eta_n} \times Q_2; \R^M)} \leq \xi,
    \]
    and for all $n\geq n_0$, all $z \in \R^M$ with $|z| \leq S$.
\end{lemma}

\begin{proof}
    Assume by contradiction that there exist $t > 0$ such that for all $k \in \N$, choosing $\xi = \frac1k$ and $n_0 = k$, we can find $n_k \geq k$, $\{ z_k \}_{k \in \N} \subset \overline{B_S}$, $\{ \psi_k \}_{k \in \N}$ with $| \psi_k | \leq S$ and $\| \psi_k \|_{L^2} \leq \frac1k$, such that
    \begin{equation}\label{eq:liminf_lemma_contradiction}
        \int_{\widehat Q_{1,\eta_{n_k}}} \int_{Q_2} W (g_{n_k}(y_1,y_2), y_2, z_k + \psi_k) \dd y_2 \, \dd y_1 < W^\mathrm{h}(z_k) - t.
    \end{equation}
    Denote with $\widetilde \psi_k$ the extension of $\psi_k$ by zero on $\Lambda_{1,\eta_{n_k}} \times Q_2$. Then $\widetilde \psi_k$ is a measurable function with $| \widetilde \psi_k | \leq S$, which implies that we can apply Lemma \ref{lemma:uniform_replacement} with $p = q = z_k + \widetilde \psi_k$, since these are bounded by $2S$. Using this, together with $W \leq C_{2S}$ on $\Lambda_{1,\eta_{n_k}} \times Q_2$ and $| \Lambda_{1,\eta_{n_k}} | \leq C \eta_{n_k} / \delta_{n_k} $, we get from \eqref{eq:liminf_lemma_contradiction} that 
    \begin{equation}\label{eq:liminf_lemma_contradiction_2}
        \int_{Q_1} \int_{Q_2} W(y_1,y_2,z_k + \widetilde \psi_k) \dd y_2 \, \dd y_1 < W^\mathrm{h}(z_k) - t + \omega_{2S} \left( C \frac{\eta_{n_k}}{\delta_{n_k}} \right) + C \frac{\eta_{n_k}}{\delta_{n_k}}.
    \end{equation}
    Since $\{ z_k \}_k \subset \overline{B_S}$, we get that up to a (non relabeled) subsequence, there exists $z_0$ such that $z_k \to z_0 \in \overline{B_S}$. Moreover, since $\widetilde \psi_k \to 0$ in $L^2 (Q_1 \times Q_2; \R^M)$, we also get $\psi_k \to 0$ pointwise for a.e. $y_1,y_2$. Since $W(y_1,y_2,\cdot)$ is continuous for a.e. $y_1,y_2$, we get
    $$ W(y_1,y_2,z_k + \widetilde \psi_k) \to W(y_1, y_2, z_0) \qquad \text{for a.e. } y_1, y_2. $$
    Thus, using this together with Fatou's lemma, \eqref{eq:liminf_lemma_contradiction_2} and the continuity of $W^\mathrm{h}$ from Lemma \ref{lemma:continuity_homogenized}, we get
    \begin{align*}
        W^\mathrm{h}(z_0) &= \int_{Q_1} \int_{Q_2} W(y_1, y_2, z_0) \dd y_2 \, \dd y_1 \\
        &= \int_{Q_1} \int_{Q_2} \lim_{k \to \infty} W(y_1, y_2, z_k + \widetilde \psi_k) \dd y_2 \, \dd y_1 \\
        &\leq \liminf_{k \to \infty} \int_{Q_1} \int_{Q_2} W(y_1, y_2, z_k + \widetilde \psi_k) \dd y_2 \, \dd y_1 \\
        &\leq \liminf_{k \to \infty} W^\mathrm{h} (z_k) - t \\
        &= W^\mathrm{h}(z_0) - t,
    \end{align*}
    which is a contradiction, and we conclude.
\end{proof}


\subsection{Estimates for sequences with uniformly bounded energies} \hfill


The goal of this section is to prove the following result, which is needed for a key-step of the proof. 

\begin{lemma}\label{lemma:L2_liminf_estimates}
    Let $\eta \ll \delta$, let $u \in W^{1,2}(\Omega; \R^M)$. For $x \in \widehat \Omega_\delta$ and $(y_1, y_2) \in \widehat Q_{1,\eta} \times Q_2$, set
    $$ \psi(x,y_1,y_2) \coloneqq \mathcal{U}_2 u(x,y_1,y_2) - u(x). $$
    Then, there exists a constant $C>0$, depending only on $Q_1$ and $Q_2$, such that
    \begin{equation}\label{eq:L2_liminf_estimate}
        \int_{\widehat \Omega_\delta} \| \psi(x, \cdot, \cdot) \|^2_{L^2 (\widehat Q_{1,\eta} \times Q_2; \R^M)} \dd x \leq C (\delta^2 + \eta^2 ) \| \nabla u \|^2_{L^2 (\Omega; \R^{N \times M})}.
    \end{equation}
\end{lemma}

\begin{proof} \textbf{Step 0: Setup.}
    We will denote with $C_1$ and $C_2$ the constants of the Poincar\'e-Wirtinger inequality in $Q_1$ and $Q_2$ respectively. We will use the identity
    \begin{equation}\label{eq:estimate_double_integral}
        \int_Q \int_Q | w(s) - w(s')|^2 \dd s \, \dd s' = 2 \int_Q | w(s) - \overline w|^2 \dd s, \qquad \overline w \coloneqq \int_Q w(s) \dd s,
    \end{equation}
    which holds for every $w \in L^2 (Q; \R^M)$ with $|Q|=1$.

    Fix $\xi_1 \in \Xi_1$ and let $D \coloneqq \delta \xi_1 + \delta Q_1$. Define $v \in W^{1,2}(Q_1; \R^M)$ as $v(s) = u (\delta \xi_1 + \delta s)$. For $x \in D$ we have $\lfloor x /\delta \rfloor_{Q_1} = \xi_1$, and by definition of the unfolding operators $\mathcal{U}_1$ and $\mathcal{U}_2$ we get
    $$ u(x) = v \left( \left\{ \frac{x}{\delta} \right\}_{Q_1} \right), \qquad \mathcal{U}_1 u(x,y_1) = v(y_1), \qquad \mathcal{U}_2 u (x, y_1, y_2) = v (g(y_1,y_2)), $$
    where the last identity holds for $y_1 \in \widehat{Q}_{1,\eta}$. Note that
    $$ \nabla v(s) = \delta \nabla u (\delta \xi_1 + \delta s) \implies \int_{Q_1} | \nabla v|^2 \dd s = \delta^{2 - N} \int_D | \nabla u|^2 \dd x. $$
    
    \textbf{Step 1: Estimate of $\mathcal{U}_1 u - u$.}
    Using the change of variables $x = \delta \xi_1 + \delta y_1'$, together with \eqref{eq:estimate_double_integral} and Poincar\'e-Wirtinger inequality on $Q_1$, we get
    \begin{align*}
        \int_D \| \mathcal{U}_1 u(x, \cdot) - u(x) \|^2_{L^2 (Q_1; \R^M)} \dd x &= \int_D \int_{Q_1} | \mathcal{U}_1 u(x,y_1) - u(x) |^2 \dd y_1 \, \dd x \\
        &= \int_D \int_{Q_1} \left| v(y_1) - v\left(\left\{ \frac{x}{\delta} \right\}_{Q_1}\right) \right|^2 \dd y_1 \, \dd x \\
        &= \delta^N \int_{Q_1} \int_{Q_1} | v(y_1) - v(y_1') |^2 \dd y_1\, \dd y_1' \\
        &= 2 \delta^N \int_{Q_1} | v(y_1) - \overline{v} |^2 \dd y_1 \\
        &\leq 2 C_1 \delta^N \int_{Q_1} | \nabla v (y_1) |^2 \dd y_1 \\
        &= 2 C_1 \delta^2 \int_D | \nabla u(x) |^2 \dd x.
    \end{align*}
    
    \textbf{Step 2: Estimate of $\mathcal{U}_2 u - \mathcal{U}_1 u$.}
    Fix $x \in D$. For $\xi_2 \in \Xi_2$, define
    $$ w_{\xi_2}(s) \coloneqq v \left( \frac{\eta}{\delta} \xi_2 + \frac{\eta}{\delta} s \right)\qquad s \in Q_2. $$
    Using in each cell the change of variables
    $$ y_1 = \frac{\eta}{\delta} \xi_2 + \frac{\eta}{\delta} y_2', $$
    and using \eqref{eq:estimate_double_integral} and Poincar\'e-Wirtinger inequality on $Q_2$, together with
    $$ v(g(y_1,y_2)) = w_{\xi_2}(y_2), \qquad v(y_1) = w_{\xi_2} (y_2'), $$
    we estimate
    \begin{align*}
        \int_{\widehat Q_{1,\eta}} \| \mathcal{U}_2 u (x,y_1, \cdot) &- \mathcal{U}_1 u(x,y_1) \|^2_{L^2 (Q_2; \R^M)} \dd y_1 \\
        &= \int_{\widehat Q_{1,\eta}} \int_{Q_2} | \mathcal{U}_2 u(x,y_1,y_2) - \mathcal{U}_1 u(x,y_1) |^2 \dd y_2 \, \dd y_1 \\
        &= \int_{\widehat Q_{1,\eta}} \int_{Q_2} \left| v(g(y_1,y_2)) - v(y_1) \right|^2 \dd y_2 \, \dd y_1 \\
        &= \sum_{\xi_2 \in \Xi_2} \int_{\frac{\eta}{\delta} \xi_2 + \frac{\eta}{\delta} Q_2} \int_{Q_2} \left| v(g(y_1,y_2)) - v(y_1) \right|^2 \dd y_2 \, \dd y_1 \\
        &= \sum_{\xi_2 \in \Xi_2} \left( \frac{\eta}{\delta} \right)^N \int_{Q_2} \int_{Q_2} \left| w_{\xi_2}(y_2) - w_{\xi_2}(y_2') \right|^2 \dd y_2 \, \dd y_2'\\
        &= \sum_{\xi_2 \in \Xi_2} 2 \left( \frac{\eta}{\delta} \right)^N \int_{Q_2} \left| w_{\xi_2}(y_2) - \overline{w_{\xi_2}} \right|^2 \dd y_2 \\
        &\leq \sum_{\xi_2 \in \Xi_2} 2 C_2 \left( \frac{\eta}{\delta} \right)^N \int_{Q_2} \left| \nabla w_{\xi_2}(y_2) \right|^2 \dd y_2\\
        &= 2 C_2 \left( \frac{\eta}{\delta} \right)^2 \sum_{\xi_2 \in \Xi_2} \int_{\frac{\eta}{\delta} \xi_2 + \frac{\eta}{\delta} Q_2} | \nabla v(y_1) |^2 \dd y_1 \\
        &\leq 2 C_2 \left( \frac{\eta}{\delta} \right)^2 \int_{Q_1} | \nabla v(y_1) |^2 \dd y_1 \\
        &= 2 C_2 \eta^2 \delta^{-N} \int_D | \nabla u(x) |^2 \dd x.
    \end{align*}
    Integrating over $x \in D$, and using that $|D| = \delta^N |Q_1| = \delta^N$, we find 
    $$ \int_D \| \mathcal{U}_2 u(x,\cdot,\cdot) - \mathcal{U}_1 u(x,\cdot) \|^2_{L^2(\widehat Q_{1,\eta} \times Q_2; \R^M)} \dd x \leq 2 C_2 \eta^2 \int_D | \nabla u(x) |^2 \dd x. $$
    
    \textbf{Step 3: Conclusion.}
    Since
    $$ | \psi|^2 = | \mathcal{U}_2 u - u |^2 \leq 2 | \mathcal{U}_2 u - \mathcal{U}_1 u|^2 + 2 | \mathcal{U}_1 u - u |^2, $$
    steps 1 and 2 give
    \begin{equation}\label{eq:lemma_liminf_localized_inter}
        \int_D \| \psi(x, \cdot, \cdot) \|^2_{L^2(\widehat Q_{1,\eta} \times Q_2; \R^M)} \dd x \leq 4 (C_1 \delta^2 + C_2 \eta^2) \int_D | \nabla u(x) |^2 \dd x.
    \end{equation}
    Denote $C \coloneqq 4 \max\{C_1, C_2\}$. Then, summing \eqref{eq:lemma_liminf_localized_inter} over $\xi \in \Xi_1$, we get \eqref{eq:L2_liminf_estimate}.

    Finally, \eqref{eq:L2_liminf_estimate_gradient} follows from \eqref{eq:L2_liminf_estimate}.
\end{proof}

\begin{corollary}\label{cor:L2_liminf_estimates}
Let $\{ u_n \}_{n \in \N} \subset W^{1,2} (\Omega; \R^M)$ be such that $\sup_n F_n (u_n) \leq C$. Then,\begin{equation}\label{eq:L2_liminf_estimate_gradient}
        \int_{\widehat \Omega_{\delta_n}} \| \psi_n(x, \cdot, \cdot) \|^2_{L^2 (\widehat Q_{1,\eta_n} \times Q_2; \R^M)} \dd x \leq C \frac{\delta_n^2 + \eta_n^2}{\varepsilon_n},
    \end{equation}
    where $\psi_n$ is defined as above with $u = u_n$, $\delta = \delta_n$ and $\eta = \eta_n$.
\end{corollary}

\begin{proof}
Note that 
\[
    \| \nabla u_n \|^2_{L^2 (\Omega; \R^{N \times M})} = \int_\Omega | \nabla u_n |^2 \dd x = \cfrac{1}{\varepsilon_n} \int_\Omega \varepsilon_n | \nabla u_n |^2 \dd x \leq \cfrac{1}{\varepsilon_n} F_n (u_n) \leq \cfrac{C}{\varepsilon_n}.
\]
Thus, the result follows from \eqref{eq:L2_liminf_estimate}.
\end{proof}


\subsection{Stability of surface tension}\label{subsec:stability}

Let $V \colon \R^M \to [0,+\infty)$ be a continuous function.
Define the surface tension of $V$ as
\begin{equation}\label{eq:sigma_V}
    \sigma(V) \coloneqq \inf \left\{ \int_{-1}^1 2\sqrt{V(\gamma(s))} | \gamma'(s)| \dd s : \gamma \in \mathrm{Lip}([-1,1];\R^M),\ \gamma(-1)=a,\ \gamma(1)=b \right\}.
\end{equation}

We now show a continuity property of the surface tension with respect to the type of perturbations that we will consider in the liminf inequality.

\begin{lemma}\label{lemma:stability_sigma}
Let $V \colon \R^M \to [0,+\infty)$ be a continuous function such that $V(z) = 0$ if and only if $z \in \{a,b\}$, and such that there exists $R > 0$ with $V(z) \geq |z|/R$ for all $|z|\geq R$.
For $t \in (0,1)$, we define $V_t \coloneqq ( V - t)^+$.
    Then, it holds that 
    $$ \lim_{t \to 0^+} \sigma(V_t) = \sigma(V).$$
\end{lemma}

\begin{figure}
    \centering
    \input{fig_curve_cutting}
    \caption{Left: a one-dimensional section of $V$ and of $V_t=(V-t)^+$; the sub-level set $\{V\le\sqrt t\}$ is contained in small compact sets around the wells.
    Right: the pieces of $\gamma$ before its last exit from $B(a,\rho)$ and after its first entry into $B(b,\rho)$ are replaced by segments; along the remaining piece $V_t\ge(1-\sqrt t)\,V$.}
    \label{fig:curve_cutting}
\end{figure}

\begin{proof}
    Since $V_t \leq V$, we have $\sigma(V_t) \leq \sigma(V)$. Moreover, since $t \to \sigma(V_t)$ is non-decreasing, we have that the limit as $t \to 0^+$ exists and
    \begin{equation}\label{eq:geodesics_first_inequality}
        \lim_{t \to 0^+} \sigma(V_t) \leq \sigma(V).
    \end{equation}

    \textbf{Step 1: Shrinking of sub-level sets.}
    For $s \in (0,1)$, let
    $$ \rho(s) \coloneqq \sup \{ \mathrm{dist}(z, \{a,b\}) : V(z) \leq s \}. $$
    If $|z| \geq R$, then $V(z) \geq 1 > s$. Therefore, $\{ V \leq s \} \subset \overline{B_R}$ is compact, and $\rho(s) < \infty$.

    We claim that $\rho(s) \to 0$ as $s \to 0^+$. Assume by contradiction that there exist $c > 0$ and $\{ z_k \}_{k \in \N} \subset \overline{B_R}$ with $V(z_k) \to 0 $ and $\mathrm{dist}(z_k, \{a,b\}) \geq c$. Up to a (non relabeled) subsequence, there exists $z_0 \in \overline{B_R}$ such that $z_k \to z_0$, and by continuity $V(z_0) = 0$, which implies $z_0 \in \{a,b\}$. By continuity it also holds $\mathrm{dist}(z_0, \{a,b\}) \geq c$, which gives a contradiction.

    We define now
    $$ \kappa (\rho) \coloneqq \max \{ \sqrt{V(z)} : z \in \overline{B(a,\rho)} \cup \overline{B(b,\rho)} \}, $$
    which tends to zero as $\rho \to 0^+$, since $V$ is continuous and $V(a)=V(b)=0$.\\

    \textbf{Step 2: Comparison of the energies.}
    Take $t \in (0,1)$ small enough such that
    $$ \rho \coloneqq \rho (\sqrt{t}) < \frac{|a-b|}{2}. $$
    Let $\gamma \in \mathrm{Lip}([-1,1]; \R^M)$ with $\gamma(-1)=a$ and $\gamma(1)=b$, and define
    $$ s_a \coloneqq \max \{ s \in [-1,1] : \gamma(s) \in \overline{B(a,\rho)} \}, \qquad s_b \coloneqq \min \{ s \in [s_a,1] : \gamma(s) \in \overline{B(b,\rho)} \}. $$
    Since the two balls are disjoint, we get $s_a < s_b$. In particular, by their definition, we get that for all $s \in (s_a,s_b)$ we have $\mathrm{dist}(\gamma(s), \{a,b\}) > \rho = \rho(\sqrt{t})$. By the definition of $\rho$, this implies $V(\gamma(s)) > \sqrt{t}$. This then implies $t < \sqrt{t}\, V(\gamma(s))$, and therefore
    $$ V_t(\gamma(s)) = V(\gamma(s)) - t \geq (1- \sqrt{t}) V(\gamma(s)) \qquad \text{for all } s \in (s_a, s_b). $$
    Hence, 
    \begin{align}
        \int_{-1}^1 2 \sqrt{V_t (\gamma(s))} | \gamma'(s) | \dd s &\geq \int_{s_a}^{s_b} 2 \sqrt{V_t (\gamma(s))} | \gamma'(s) | \dd s \nonumber \\
        &\geq \sqrt{1 - \sqrt{t}} \int_{s_a}^{s_b} 2 \sqrt{V (\gamma(s))} | \gamma'(s) | \dd s. \label{eq:geodesics:lower_bound}
    \end{align}

    \textbf{Step 3: Construction of admissible curve.}
    Consider now the curve $\widetilde \gamma$, suitably reparametrized on $[-1,1]$, obtained by concatenating the segment from $a$ to $\gamma(s_a)$, the restriction of $\gamma$ to $[s_a,s_b]$, and the segment from $\gamma(s_b)$ to $b$, namely
    \begin{gather*}
    \widetilde{\gamma}(s) \coloneqq \begin{cases}
        \displaystyle a + \frac{s+1}{s_a + 1} (\gamma(s_a) - a) \quad &s \in \left[ -1, s_a \right), \\
        \displaystyle\gamma \left( s \right) \quad &s \in \left[ s_a, s_b \right], \\
        \displaystyle\gamma(s_b) + \frac{s - s_b}{1 - s_b} (b - \gamma(s_b)) \quad &s \in \left( s_b, 1 \right].
    \end{cases}
\end{gather*}
This curve is an admissible competitor for $\sigma(V)$, and the two segments have length at most $\rho$ and lie in $\overline{B(a,\rho)}$ and $\overline{B(b,\rho)}$ respectively, which we can assume disjoint (see Figure \ref{fig:curve_cutting}). Therefore, using \eqref{eq:geodesics:lower_bound} we get
\begin{align*}
    \sigma(V) &\leq \int_{s_a}^{s_b} 2 \sqrt{V(\gamma(s))} |\gamma'(s) | \dd s + 4 \rho \kappa(\rho) \\
    &\leq \frac{1}{\sqrt{1 - \sqrt{t}}} \int_{-1}^1 2 \sqrt{V_t (\gamma(s))} | \gamma'(s)| \dd s + 4 \rho \kappa (\rho).
\end{align*}
Taking the infimum over $\gamma$, we obtain
$$ \sigma(V) \leq \frac{1}{\sqrt{1 - \sqrt{t}}} \sigma(V_t) + 4 \rho(\sqrt{t}) \kappa (\rho(\sqrt{t})). $$
Letting $t \to 0^+$ and using Step 1, we get
\begin{equation*}
    \sigma(V) \leq \lim_{t\to 0^+} \sigma(V_t), 
\end{equation*}
and, together with \eqref{eq:geodesics_first_inequality}, we conclude.
\end{proof}

\begin{remark}
    The proof does not depend on the existence of a minimum. This, while outside of the scope of this paper, can be proved following the same techniques as in \cite[Theorem 2.6]{zuniga2016heteroclinic} (see also \cite[Section 3.3]{CriFonGan_fixed_Sup}).
\end{remark}


\section{Compactness}

In this section we prove Theorem \ref{thm:compactness}.
\begin{proof}
    Let $\{u_n\}_{n \in \N} \subset W^{1,2}(\Omega; \R^M)$ be a sequence such that
    $$ \sup_{n \in \N} F_n (u_n) \leq C < +\infty. $$
    Let $W_0 \colon \R^M \to [0,+\infty)$ be the continuous double-well potential given by \ref{itm:1_H3}, and let $R > 0$ be the constant given by \ref{itm:1_H5}. Up to increasing $R$, we can assume $R > \max\{|a|,|b| \}$. 

    Take a function $\varphi \in C^\infty ([0,+\infty))$ such that $\varphi \equiv 1$ for $t \in [0,R)$ and $\varphi \equiv 0$ for $t \geq 2R$. We now define the modified potential $\mathcal{W} \colon \R^M \to [0,+\infty)$ as
    $$ \mathcal{W}(z) \coloneqq \varphi (|z|) W_0 (z) + (1 - \varphi(|z|) ) \frac{1}{R} |z|. $$
    This potential satisfies \ref{itm:1_H3} and \ref{itm:1_H5}, and since $W_0 (z) \leq W (y_1,y_2,z)$ for a.e. $y_1 \in Q_1$, $y_2 \in Q_2$, we get
    $$ \sup_{n \in \N} \int_\Omega \left[ \frac{1}{\varepsilon_n}  \mathcal{W} (u_n) + \varepsilon_n | \nabla u_n |^2 \right] \dd x \leq \sup_{n \in \N} F_n (u_n) \leq C < +\infty. $$
    This allows us to use the methods of \cite[Section 4]{fonseca1989gradient} in order to extract a subsequence $\{u_{n_k}\}_{k \in \N} \subset W^{1,2}(\Omega; \R^M)$ such that $u_{n_k} \to u \in \mathrm{BV}(\Omega; \{a, b\})$ strongly in $L^1 (\Omega; \R^M)$. 
\end{proof}


\section{Liminf inequality}\label{sec:liminf}

The goal of this section is to prove the following proposition.
\begin{proposition}\label{prop:liminf}
    Let $\{u_n\}_{n \in \N} \subset W^{1,2}(\Omega; \R^M)$ be a sequence such that $u_n \to u \in \mathrm{BV}(\Omega; \{ a, b \})$ strongly in $L^1 (\Omega; \R^M)$. Then, it holds that
    $$ \liminf_{n \to \infty} F_n (u_n) \geq F_\infty(u). $$
\end{proposition}
\begin{proof}
\textbf{Step 1: Reduction to uniformly bounded sequences.}
Let $S > 0$ such that $S > 2R$, where $R > 0$ is the constant appearing in \ref{itm:1_H5}. Up to increasing $R$, we can assume that $R > \max \{ |a|,|b| \}$.

Let $\varphi_S \colon [0, +\infty) \to [0,1]$ be a smooth cut-off function such that
\begin{gather*}
    \varphi_S (t) \begin{cases}
        = 1 \qquad &t \in \left[ 0,\frac{S}{2} \right], \\
        \in (0,1) \qquad &t \in \left( \frac{S}{2}, S \right),\\
        = 0 \qquad &t \in [S,+\infty).
    \end{cases}
\end{gather*}
We now modify $W$ such that it is linear outside of a ball of radius $S$: to do that, we define $\widetilde{W}_S \colon Q_1 \times Q_2 \times \R^M \to [0,+\infty) $ as
\begin{equation*}
    \widetilde{W}_S (y_1, y_2, z) \coloneqq \varphi_S (|z|) W(y_1, y_2, z) + (1 - \varphi_S (|z|)) \cfrac{|z|}{R}.
\end{equation*}
Observe that this function still satisfies all the needed assumptions, that is \ref{itm:1_H1}, \ref{itm:1_H2}, \ref{itm:1_H3} and \ref{itm:1_H4}.
In this way, we have
\begin{gather*}
    \widetilde{W}_S (y_1, y_2, z) \begin{cases}
        = W(y_1, y_2, z) \qquad &|z| \leq \frac{S}{2}, \\
        \geq \frac{|z|}{R} \qquad &|z| \geq R,\\
        = \frac{|z|}{R} \qquad &|z| \geq S.
    \end{cases}
\end{gather*}
We can then define
\begin{equation*}
    \widetilde{F}_{n,S} (v) \coloneqq \int_\Omega \left[ \cfrac{1}{\varepsilon_n} \widetilde{W}_S \Bigg( \!\bigg\{\! \frac{x}{\delta_n} \! \bigg\}_{Q_1}\!, \bigg\{\! \frac{\delta_n}{\eta_n}\!\bigg\{ \!\frac{x}{\delta_n} \!\bigg\}_{Q_1} \!\bigg\}_{Q_2}\!, v \!\Bigg) + \varepsilon_n | \nabla v |^2 \right] \dd x.
\end{equation*}
Our goal is to now prove that it is enough to prove the statement of the proposition for a sequence uniformly bounded in $L^\infty$. This means that, if
\begin{equation}\label{eq:reduced_liminf}
    F_\infty (u) \leq \liminf_{n \to \infty} \widetilde{F}_{n,S} (v_n),
\end{equation}
holds for all $u \in \mathrm{BV}(\Omega;\{a,b\})$ and all sequences $\{v_n\}_n \subset W^{1,2}(\Omega; \R^M)$ such that $v_n \to u$ in $L^1 (\Omega;\R^M)$ and $\| v_n \|_\infty \leq S$, then
\begin{equation}\label{eq:liminf}
    F_\infty (u) \leq \liminf_{n \to \infty} F_n (u_n),
\end{equation}
holds for all sequences $\{u_n\}_n \subset W^{1,2}(\Omega; \R^M)$ such that $u_n \to u$ in $L^1 (\Omega;\R^M)$. \\
Let us then assume that \eqref{eq:reduced_liminf} holds, and fix $u \in \mathrm{BV} (\Omega; \{ a,b\})$ and a sequence $\{u_n\}_n \subset W^{1,2}(\Omega; \R^M)$ with $u_n \to u$ in $L^1 (\Omega;\R^M)$. 
We define the truncation operator $\mathcal{T}_S \colon W^{1,2}(\Omega; \R^M) \to L^\infty (\Omega;\R^M) \cap W^{1,2} (\Omega; \R^M)$ as
\begin{gather*}
    \mathcal{T}_S \phi(x) = \begin{cases}
        \phi(x) \qquad &|\phi(x)| \leq S, \\
        S \frac{\phi(x)}{|\phi(x)|} \qquad &|\phi(x)| > S.
    \end{cases}
\end{gather*}
This definition implies
\begin{equation}\label{eq:gradient_truncation}
	\|\mathcal{T}_S \phi \|_\infty \leq S, \qquad | \nabla \mathcal{T}_S \phi (x) | \leq | \nabla \phi (x)|.
\end{equation}
Using this operator, we define $v_n = \mathcal{T}_{S} u_n$, which is going to be a sequence uniformly bounded in $L^\infty$. Therefore, \eqref{eq:reduced_liminf} holds if we prove that $v_n \to u$ in $L^1 (\Omega;\R^M)$. Using the triangle inequality it is enough to prove that $v_n - u_n \to 0$ in $L^1 (\Omega;\R^M)$, as
$$ \| v_n - u \|_{L^1 (\Omega;\R^M)} \leq \| v_n - u_n \|_{L^1 (\Omega;\R^M)} + \| u_n - u \|_{L^1 (\Omega;\R^M)}. $$
The definition of truncation operator yields that:
\begin{align}\label{eq:reduced_liminf_2}
    \| u_n - v_n \|_{L^1 (\Omega;\R^M)} &= \int_\Omega | u_n - v_n | \dd x \nonumber \\
    &= \int_{\{ | u_n|\leq S \}} |u_n - u_n|\dd x + \int_{\{ | u_n| > S \}} |u_n - v_n|\dd x \nonumber \\
    &= \int_{\{ | u_n| > S \}} |u_n - v_n|\dd x.
\end{align}
Now let $\Lambda_n \coloneqq \{ x \in \Omega : | u_n (x) | > S \}$. Using Chebychev's inequality and assumption \ref{itm:1_H5} we have
\begin{equation*}
    | \Lambda_n | \leq \frac{1}{S} \! \int_{\Lambda_n} \! |u_n| \dd x \leq \frac{R}{S}\! \int_{\Lambda_n}\! W\! \Bigg( \!\bigg\{\! \frac{x}{\delta_n} \! \bigg\}_{Q_1}\!, \bigg\{\! \frac{\delta_n}{\eta_n}\!\bigg\{ \!\frac{x}{\delta_n} \!\bigg\}_{Q_1} \!\bigg\}_{Q_2}\!, u_n \!\Bigg)\! \dd x \leq C \frac{R}{S} \varepsilon_n.
\end{equation*}
Going back to \eqref{eq:reduced_liminf_2}, using triangle inequality and Assumption \ref{itm:1_H5} we now have
\begin{align*}
    \| u_n - v_n \|_{L^1 (\Omega;\R^M)} &= \int_{\Lambda_n} |u_n - v_n|\dd x \\
    &\leq \int_{\Lambda_n} | u_n | \dd x + \int_{\Lambda_n} | v_n | \dd x \\
    &\leq R \int_{\Lambda_n} W \Bigg( \!\bigg\{\! \frac{x}{\delta_n} \! \bigg\}_{Q_1}\!, \bigg\{\! \frac{\delta_n}{\eta_n}\!\bigg\{ \!\frac{x}{\delta_n} \!\bigg\}_{Q_1} \!\bigg\}_{Q_2}\!, u_n \!\Bigg) \dd x + \int_{\Lambda_n} S \dd x \\
    &\leq R \int_\Omega W \Bigg( \!\bigg\{\! \frac{x}{\delta_n} \! \bigg\}_{Q_1}\!, \bigg\{\! \frac{\delta_n}{\eta_n}\!\bigg\{ \!\frac{x}{\delta_n} \!\bigg\}_{Q_1} \!\bigg\}_{Q_2}\!, u_n \!\Bigg) \dd x + S | \Lambda_n | \\
    &\leq c R \varepsilon_n + c R \varepsilon_n = C \varepsilon_n,
\end{align*}
and therefore we conclude. \\
In order to prove \eqref{eq:liminf} we need to prove that $\widetilde{F}_{n,S} (v_n) \leq F_n (u_n)$. This, together with \eqref{eq:reduced_liminf}, will let us conclude.\\
From the definition of $\Lambda_n$ and the truncation operator we have that for $x \in \Lambda_n$ it holds
$$ |u_n(x)| > S = | \mathcal{T}_{S} u_n (x) | = |v_n (x)|. $$
Therefore, for $x \in \Lambda_n$ we have
\begin{align*}
	\widetilde{W}_S \!\Bigg( \!\bigg\{\! \frac{x}{\delta_n} \! \bigg\}_{Q_1}\!, \bigg\{\! \frac{\delta_n}{\eta_n}\!\bigg\{ \!\frac{x}{\delta_n} \!\bigg\}_{Q_1} \!\bigg\}_{Q_2}\!, u_n(x) \!\Bigg) &= \frac{|u_n(x)|}{R} \\
	&\geq \frac{|v_n (x)|}{R} = \widetilde{W}_S\! \Bigg( \!\bigg\{\! \frac{x}{\delta_n} \! \bigg\}_{Q_1}\!, \bigg\{\! \frac{\delta_n}{\eta_n}\!\bigg\{ \!\frac{x}{\delta_n} \!\bigg\}_{Q_1} \!\bigg\}_{Q_2}\!, v_n(x) \!\Bigg).
\end{align*}
For $x \in \Omega \setminus \Lambda_n$, it holds that $v_n(x) = \mathcal{T}_{S} u_n (x) = u_n (x)$, therefore
$$ \widetilde{W}_S \Bigg( \!\bigg\{\! \frac{x}{\delta_n} \! \bigg\}_{Q_1}\!, \bigg\{\! \frac{\delta_n}{\eta_n}\!\bigg\{ \!\frac{x}{\delta_n} \!\bigg\}_{Q_1} \!\bigg\}_{Q_2}\!, u_n(x) \!\Bigg) = \widetilde{W}_S \Bigg( \!\bigg\{\! \frac{x}{\delta_n} \! \bigg\}_{Q_1}\!, \bigg\{\! \frac{\delta_n}{\eta_n}\!\bigg\{ \!\frac{x}{\delta_n} \!\bigg\}_{Q_1} \!\bigg\}_{Q_2}\!, v_n(x) \!\Bigg). $$
This implies that
\begin{align*}
	\widetilde{W}_S \Bigg( \!\bigg\{\! \frac{x}{\delta_n} \! \bigg\}_{Q_1}\!, \bigg\{\! \frac{\delta_n}{\eta_n}\!\bigg\{ \!\frac{x}{\delta_n} \!\bigg\}_{Q_1} \!\bigg\}_{Q_2}\!, v_n(x) \!\Bigg) &\leq \widetilde{W}_S \Bigg( \!\bigg\{\! \frac{x}{\delta_n} \! \bigg\}_{Q_1}\!, \bigg\{\! \frac{\delta_n}{\eta_n}\!\bigg\{ \!\frac{x}{\delta_n} \!\bigg\}_{Q_1} \!\bigg\}_{Q_2}\!, u_n(x) \!\Bigg) \\
	&\leq W \Bigg( \!\bigg\{\! \frac{x}{\delta_n} \! \bigg\}_{Q_1}\!, \bigg\{\! \frac{\delta_n}{\eta_n}\!\bigg\{ \!\frac{x}{\delta_n} \!\bigg\}_{Q_1} \!\bigg\}_{Q_2}\!, u_n(x) \!\Bigg),
\end{align*}
which, together with \eqref{eq:gradient_truncation} allows us to conclude.\\
Therefore from now on we assume that $W$ has linear growth at infinity and that $\| u_n \|_\infty \leq S$.

Note that, thanks to \ref{itm:1_H5}, for $|z|\geq R$ it holds that $\widetilde{W}_S^\mathrm{h} (z) \geq \frac1R |z|$. Inside the ball $B_{S/2}$ we have by definition $\widetilde{W}_S = W$, and in general it holds $\widetilde{W}_S \leq W$ everywhere. These two facts together imply that if a curve connecting the wells reaches $\partial B_{S/2}$, then it must cross the annulus $R \leq |z| \leq S/2$ twice. 

Let us denote by $s_0$ a point where $|\gamma(s_0)|=S/2$. Since $|a|,|b|<R$, we can define
$$ s_1 \coloneqq \max \{ s < s_0 : | \gamma(s)|=R \}, \qquad s_2 \coloneqq \min \{ s > s_0 : | \gamma(s)|=R \}. $$
Then we have $|\gamma(s)| \geq R$ in $[s_1,s_2]$. Using the change of variables $r = | \gamma(s)|$ we get
\begin{align*}
    \int_{s_1}^{s_2} 2 \sqrt{\widetilde{W}_S^\mathrm{h} (\gamma(s))} | \gamma'(s)| \dd s &\geq \int_{s_1}^{s_0} 2 \sqrt{\frac{|\gamma(s)|}{R}} \frac{\dd}{\dd s} |\gamma(s)| \dd s - \int_{s_0}^{s_2} 2 \sqrt{\frac{|\gamma(s)|}{R}} \frac{\dd}{\dd s} |\gamma(s)| \dd s \\
    &\geq 2 \int_R^{S/2} 2 \sqrt{\frac{r}{R}} \dd r \\
    &= \frac{8}{3 \sqrt{R}} \left( \left(\frac{S}{2}\right)^{3/2} - R^{3/2} \right). 
\end{align*}
Since the right-hand side is bigger than $\sigma^\mathrm{h}$ for $S$ large enough, we conclude that the truncated potential gives the same surface tension in the limit, up to choosing the truncation level large enough.\\

\textbf{Step 2: Lower bound for the potential term.}
Fix $t \in (0,1)$, and let
\[
V_t(z) \coloneqq (W^\mathrm{h}(z) - t)^+.
\]
Let $\xi > 0$ and $n_0 \in \N$ be given by Lemma \ref{lemma:liminf_lemma_amazing}, applied with $2S$ in place of $S$ and with the value of $t$ fixed above. 

For $x \in \widehat \Omega_{\delta_n}$ and $(y_1,y_2) \in \widehat Q_{1,\eta_n} \times Q_2$, let
\[
\psi_n (x,y_1,y_2) \coloneqq \mathcal{U}_2 u_n(x,y_1,y_2) - u_n(x).
\]
Define the set of good points
$$ G_n \coloneqq \{ x \in \widehat \Omega_{\delta_n} : \| \psi_n(x,\cdot,\cdot) \|_{L^2(\widehat Q_{1,\eta_n} \times Q_2; \R^M)} \leq \xi \}. $$
Using Chebychev's inequality, together with Corollary \ref{cor:L2_liminf_estimates} and $|\Lambda_{\delta_n}| \leq C \delta_n$, we get
\begin{equation}\label{eq:liminf_boundary_estimate}
    | \Omega \setminus G_n | \leq | \Lambda_{\delta_n} | + \frac{1}{\xi^2} \int_{\widehat \Omega_{\delta_n}} \| \psi_n (x,\cdot,\cdot)\|^2_{L^2(\widehat Q_{1,\eta_n} \times Q_2; \R^M)} \dd x \leq C \delta_n + \frac{C}{\xi^2} \frac{\delta_n^2 + \eta_n^2}{\varepsilon_n}.
\end{equation}
Let now $n \geq n_0$ and $x \in G_n$. Since $|u_n(x)|\leq S$, we get $| \psi_n (x,\cdot, \cdot)| \leq 2S$; since on $G_n$ we also have
\[
\| \psi_n (x,\cdot,\cdot)\|_{L^2(\widehat Q_{1,\eta_n} \times Q_2; \R^M)} \leq \xi,
\]
we can apply Lemma \ref{lemma:liminf_lemma_amazing} to get
$$ \int_{\widehat Q_{1,\eta_n}} \int_{Q_2} W(g_n(y_1,y_2), y_2, u_n(x) + \psi_n(x,y_1,y_2) ) \dd y_2 \, \dd y_1 \geq W^\mathrm{h}(u_n(x)) - t. $$
Since the left-hand side is also non-negative, we also have
\begin{align}
    \int_{\widehat Q_{1,\eta_n}} \int_{Q_2} W(g_n(y_1,y_2), y_2, u_n(x) + \psi_n(x,y_1,y_2) ) \dd y_2 \, \dd y_1 &\geq ( W^\mathrm{h}(u_n(x)) - t)^+ \nonumber \\
    &= V_t(u_n(x)). \label{eq:liminf_lower_estimate}
\end{align}
Using now Lemma \ref{lemma:needed}, together with \eqref{eq:liminf_lower_estimate} and the non-negativity of $W$, and using that 
$$ 0 \leq V_t \leq W^\mathrm{h} \leq C_S \quad \text{in }\overline{B_S},$$
we get
\begin{align*}
    \int_{\Omega} W \Bigg( \bigg\{ \frac{x}{\delta_n}  \bigg\}_{Q_1}, \bigg\{ \frac{\delta_n}{\eta_n}\bigg\{ \frac{x}{\delta_n} \bigg\}_{Q_1} \bigg\}_{Q_2}, u_n \Bigg) \dd x &\geq \int_{\widehat\Omega_{\delta_n}} \int_{\widehat Q_{1,\eta_n}} \int_{Q_2} W \left( g_n(y_1,y_2), y_2, \mathcal{U}_2 u_n  \right) \dd y_2 \dd y_1 \dd x \\
    &\geq \int_{G_n} \int_{\widehat Q_{1,\eta_n}} \int_{Q_2} W \left( g_n(y_1,y_2), y_2, \mathcal{U}_2 u_n  \right) \dd y_2 \dd y_1 \dd x \\
    &\geq \int_{G_n} V_t (u_n(x)) \dd x \\
    &\geq \int_\Omega V_t (u_n(x)) \dd x - C_S | \Omega \setminus G_n|.
\end{align*}
Using \eqref{eq:liminf_boundary_estimate} we obtain
$$ F_n (u_n) \geq \int_\Omega \left[ \frac{1}{\varepsilon_n} V_t(u_n(x)) + \varepsilon_n | \nabla u_n(x)|^2 \right] \dd x - C \left( \frac{\delta_n}{\varepsilon_n} + \frac{\delta_n^2 + \eta_n^2}{\xi^2 \varepsilon_n^2} \right). $$
Taking the liminf as $n \to \infty$, and using $\eta_n \ll \delta_n \ll \varepsilon_n$, we finally get
\begin{equation}\label{eq:liminf_end_step2}
    \liminf_{n \to \infty} F_n (u_n) \geq \liminf_{n \to \infty} \int_\Omega \left[ \frac{1}{\varepsilon_n} V_t(u_n(x)) + \varepsilon_n | \nabla u_n(x)|^2 \right] \dd x.
\end{equation}

\textbf{Step 3: The Modica-Mortola argument for $V_t$.}
We claim that
\begin{equation}\label{eq:liminf_claim_step3}
    \liminf_{n \to \infty} \int_\Omega \left[ \frac{1}{\varepsilon_n} V_t(u_n(x)) + \varepsilon_n | \nabla u_n(x)|^2 \right] \dd x \geq \sigma_t \mathrm{Per}(\{ u=a\}; \Omega),
\end{equation}
where $\sigma_t\coloneqq \sigma(V_t)$ as in \eqref{eq:sigma_V}.

The potential $V_t$ does not depend on $n$, it is continuous by Lemma \ref{lemma:continuity_homogenized} as a composition of continuous functions, and vanishes in a whole neighborhood of $a$ and $b$. For $p \in \R^M$ we define
$$ \dd_t (p)\coloneqq \inf \left\{ \int_{-1}^1 2 \sqrt{V_t(\gamma(s))} | \gamma'(s)| \dd s : \gamma \in \mathrm{Lip}([-1,1];\R^M), \gamma(-1) = a, \gamma(1) = p \right\}, $$
and let $\sigma_t \coloneqq \dd_t (b)$. 
We prove now that $\dd_t(p)$ is Lipschitz continuous. 

Take $p_1, p_2 \in \R^M$, and take an admissible competitor $\gamma \in \mathrm{Lip} ([-1,1]; \R^M)$ for $\dd_t (p_1)$, such that $\gamma(-1) = a$ and $\gamma(1) = p_1$. We define now a curve $\gamma_0 \in \mathrm{Lip}([-1,1]; \R^M)$, such that for $s \in[-1,0]$ it is a reparameterization of $\gamma$, while for $s \in (0,1]$ it is the segment from $p_1$ to $p_2$, namely
\begin{gather*}
    \gamma_0 (s) \coloneqq \begin{cases}
        \gamma( 2s+1) \quad &s\in [-1, 0], \\
        p_1 + s (p_2 - p_1 ) \quad &s \in (0,1].
    \end{cases}
\end{gather*}
This is clearly an admissible competitor in the infimum problem defining $\dd_t (p_2)$.
Therefore we get that 
\begin{align*}
    \dd_t (p_2) &\leq \int_{-1}^1 2 \sqrt{V_t (\gamma_0 (s))} | \gamma'_0 (s)| \dd s \\
    &= \int_{-1}^0 2 \sqrt{V_t (\gamma (2s+1))} | 2\gamma' (2s+1)| \dd s + \int_0^1 2\sqrt{V_t (p_1 + s(p_2 - p_1))} |p_2 - p_1 | \dd s \\
    &\leq \int_{-1}^1 2 \sqrt{V_t (\gamma (s))} | \gamma' (s)| \dd s + \left( \sup_{s \in [0,1]} 2 \sqrt{V_t (p_1 + s(p_2 - p_1))} \right) |p_2 - p_1|.
\end{align*}
The term in parenthesis is always finite since $V_t$ is continuous.
Thanks to assumption \ref{itm:1_H4}, taking the infimum over all $\gamma \in \mathrm{Lip} ([-1,1]; \R^M)$ such that $\gamma(-1) = a$ and $\gamma(1) = p_1$, we get
$$ \dd_t (p_2) \leq \dd_t (p_1) + \max_{z \in [p_1,p_2]} 2\sqrt{V_t(z)} |p_2 - p_1|, $$
where $[p_1,p_2]$ is the segment between $p_1$ and $p_2$.
Writing the same procedure with $p_1 $ and $p_2$ swapped, we then have
\begin{equation}\label{eq:liminf:lipschitz_dt}
    | \dd_t (p_2) - \dd_t (p_1) | \leq \max_{z \in \overline{B_S}}2 \sqrt{V_t(z)} |p_2 - p_1|.
\end{equation}
In particular, $\dd_t$ is Lipschitz in $\overline{B_S}$ with Lipschitz constant
$$ L_t \coloneqq \max_{p \in \overline{B_S}} 2 \sqrt{V_t(p)}, $$
and thus $\dd_t \circ u_n \in W^{1,2}(\Omega)$.

Let now $x \in \Omega$ be a point where both $u_n$ and $\dd_t \circ u_n$ are approximately differentiable, which is the case for $\mathcal{L}^N$-a.e. $x \in \Omega$. Then, using \eqref{eq:liminf:lipschitz_dt} on the balls $B(u_n(x),r)$, and letting $r \to 0^+$, the continuity of $V_t$ gives
\begin{equation}\label{eq:liminf_geodesic_last}
    | \nabla (\dd_t\circ u_n)(x)| \leq 2 \sqrt{V_t(u_n(x))}| \nabla u_n(x)| \qquad \text{for } \mathcal{L}^N\text{-a.e. } x \in \Omega.
\end{equation}
Moreover, since $| \dd_t(u_n) - \dd_t (u) | \leq L_t | u_n - u|$, we have $\dd_t \circ u_n \to \dd_t \circ u$ in $L^1(\Omega)$ as $n \to \infty$.

Since $u \in \mathrm{BV}(\Omega; \{a,b\})$, and $\dd_t(a)= 0$, we have $\dd_t \circ u = \sigma_t \mathbbm{1}_{\{u=b\}}$.

Therefore, using Young's inequality, \eqref{eq:liminf_geodesic_last}, and the lower semi-continuity of the total variation with respect to the $L^1$ convergence, we get
\begin{align*}
    \liminf_{n \to \infty} \int_\Omega \left[ \frac{1}{\varepsilon_n} V_t (u_n) + \varepsilon_n | \nabla u_n|^2 \right] \dd x &\geq \liminf_{n \to \infty} \int_\Omega 2 \sqrt{V_t(u_n)} |\nabla u_n| \dd x\\
    &\geq \liminf_{n \to \infty} \int_\Omega | \nabla (\dd_t \circ u_n) | \dd x\\
    &\geq | D(\dd_t \circ u)|(\Omega) \\
    &= \sigma_t \mathrm{Per}(\{u=b\}; \Omega) \\
    &= \sigma_t \mathrm{Per}(\{u=a\}; \Omega),
\end{align*}
where the last equality holds since we are dealing with just two phases.\\

\textbf{Step 4: Conclusion.} Using \eqref{eq:liminf_end_step2} and \eqref{eq:liminf_claim_step3}, we get that
\[
    \liminf_{n \to \infty} F_n (u_n) \geq \sigma_t \, \mathrm{Per} (\{ u = a \}; \Omega).
\]
Taking the limit as $t\to0^+$, and using Lemma \ref{lemma:stability_sigma}, we get
\[
 \liminf_{n \to \infty} F_n (u_n) \geq \sigma^\mathrm{h} \, \mathrm{Per} (\{ u = a \}; \Omega),
\]
where $\sigma^\mathrm{h}\coloneqq \sigma(W^\mathrm{h})$, and we conclude.
\end{proof}


\section{Limsup inequality}\label{sec:limsup}

The goal of this section is to prove the following proposition.
\begin{proposition}\label{prop:limsup}
    Let $u \in \mathrm{BV}(\Omega; \{a,b\})$. Then, there exists a sequence $\{u_n\}_{n \in \N} \subset W^{1,2}(\Omega; \R^M)$, such that $u_n \to u$ strongly in $L^1(\Omega; \R^M)$ and 
    \begin{equation*}
        \limsup_{n \to \infty} F_n (u_n) \leq F_\infty (u).
    \end{equation*}
\end{proposition}

In order to prove this proposition, we need some technical results, the first of which is an approximation result for sets (see \cite[Lemma 1.15]{Leo_lecture} and \cite[Lemma 3.1]{baldo1990minimal}).
\begin{proposition}\label{prop:set_approx}
    Let $E \subset \Omega$ be a set with finite perimeter. Then, there exists a sequence of sets $\{E_n\}_n$ with $E_n \subset \R^N$ such that
    \begin{itemize}
        \item $\partial E_n \cap \Omega$ is of class $C^2$;
        \item $\mathcal{H}^{N-1} (\partial E_n \cap \partial \Omega) = 0$;
        \item $\mathbbm{1}_{E_n} \to \mathbbm{1}_E$ in $L^1 (\Omega)$;
        \item $\displaystyle \lim_{n \to \infty} F_\infty (a\mathbbm{1}_{E_n} + b \mathbbm{1}_{\Omega \setminus E_n}) = F_\infty (a\mathbbm{1}_{E} + b \mathbbm{1}_{\Omega \setminus E})$.
    \end{itemize}
\end{proposition}

We also need the following result about a geometric property of the tubular neighborhoods of a $C^2$ surface in $\R^N$.

\begin{lemma}\label{lemma:tubular}
    Let $A \subset \R^N$ be an open bounded set with $C^2$ boundary, and let $\Omega \subset \R^N$ be an open bounded set with Lipschitz boundary such that $\mathcal{H}^{N-1} (\partial \Omega \cap \partial A) = 0$. Define the function $\mathrm{s} \colon \R^N \to \R$ as
    \begin{gather*}
        \mathrm{s}(x) \coloneqq \begin{cases}
            - \mathrm{dist}(x, \partial A) \quad &x \in A \\
            \mathrm{dist}(x, \partial A) \quad &x \notin A.
        \end{cases}
    \end{gather*}
    Then $\mathrm{s}$ is Lipschitz continuous and $| \mathrm{D}\mathrm{s}(x)|=1$ for a.e. $x \in \R^N$.\\
    Moreover, define $\mathrm{S}_t \coloneqq \{ x \in \R^N : \mathrm{s}(x) = t \}$. Then it holds
    \begin{equation}
        \lim_{t \to 0} \mathcal{H}^{N-1} (S_t \cap \Omega) = \mathcal{H}^{N-1} (\partial A \cap \Omega). \label{eq:hausdorff_approx}
    \end{equation}
\end{lemma}

We will use the following classical fact on the reparameterization of curves.
 
\begin{remark}[Constant speed reparameterization]\label{rem:constant_speed}
    Let $\gamma\in\mathrm{Lip}([-1,1];\R^M)$ with length
    $$ L(\gamma)\coloneqq\int_{-1}^1|\gamma'(s)|\dd s>0, $$
    and let $V\colon\R^M\to[0,+\infty)$ be continuous. Then there exists $\widehat\gamma\in\mathrm{Lip}([-1,1];\R^M)$ with 
    $$ \widehat\gamma(\pm1)=\gamma(\pm1), \qquad \widehat\gamma([-1,1])=\gamma([-1,1]), \qquad |\widehat\gamma'(s)|=L(\gamma)/2 \quad \text{for a.e. } s \in (-1,1),$$
    and
    $$ \int_{-1}^1 2\sqrt{V(\widehat\gamma(s))}\,|\widehat\gamma'(s)|\dd s=\int_{-1}^1 2\sqrt{V(\gamma(r))}\,|\gamma'(r)|\dd r. $$
    Indeed, let 
    $$ \mathcal S(r)\coloneqq-1+\frac{2}{L(\gamma)}\int_{-1}^r|\gamma'(\varrho)|\dd\varrho,$$ which is Lipschitz and non-decreasing from $[-1,1]$ onto $[-1,1]$. 
    Since $\gamma$ is constant on every interval where $\mathcal S$ is constant, the function $\widehat\gamma(\mathcal S(r))\coloneqq\gamma(r)$ is well defined, and
    $$ |\widehat\gamma(\mathcal S(r_1))-\widehat\gamma(\mathcal S(r_2))|\leq\left|\int_{r_1}^{r_2}|\gamma'(\varrho)|\dd\varrho\right|=\frac{L(\gamma)}{2}\,|\mathcal S(r_1)-\mathcal S(r_2)|, $$
    so that $\widehat\gamma$ is Lipschitz with $|\widehat\gamma'|\leq L(\gamma)/2$ a.e. The chain rule $\gamma'(r)=\widehat\gamma'(\mathcal S(r))\,\mathcal S'(r)$, which holds for a.e. $r$, together with the change of variables $s=\mathcal S(r)$, gives the equality of the two integrals, and, with $V\equiv1$, that $L(\widehat\gamma)=L(\gamma)$. Together with $|\widehat\gamma'|\leq L(\gamma)/2$, this forces $|\widehat\gamma'|=L(\gamma)/2$ a.e. in $(-1,1)$.
\end{remark}
 
The next result is the key point in the original procedure by Modica-Mortola, as it ensures the existence of a reparameterization of the optimal curve up to a small error (see \cite[proof of Proposition 2]{modica1987gradient}, \cite[Lemma 6.3]{CriGra}). We state it for constant speed Lipschitz curves and give the proof.
 
\begin{lemma}\label{lemma:modica_mortola}
    Fix $\lambda > 0$, $\varepsilon > 0$. Let $\gamma \in \mathrm{Lip}([-1,1]; \R^M)$, with $\gamma (-1) = a$, $\gamma(1)=b$, and $|\gamma'(s)| = L(\gamma)/2$ for a.e. $s \in (-1,1)$, and set
    $$ m_\gamma \coloneqq \max_{s\in[-1,1]} W^\mathrm{h}(\gamma(s)). $$
    Then, there exists $\tau > 0$ with
    $$ \frac{L(\gamma)}{2\sqrt{\lambda+m_\gamma}} \varepsilon \leq \tau \leq \frac{L(\gamma)}{2\sqrt{\lambda}} \varepsilon, $$
    and a strictly increasing $h \in C^1 ([-\tau,\tau]; [-1,1])$ with $h(-\tau) = -1$, $h(\tau) = 1$, such that
    \begin{equation}\label{eq:reparametrization_diffeq}
        (h'(t))^2 = \frac{\lambda + W^\mathrm{h} (\gamma (h(t)))}{\varepsilon^2 | \gamma'(h(t)) |^2}
    \end{equation}
    for a.e. $t \in (-\tau, \tau)$.
    Moreover, $\gamma\circ h$ is Lipschitz, and
    \begin{equation}\label{eq:reparametrization_lip}
        |(\gamma\circ h)'(t)| = |\gamma'(h(t))|\,|h'(t)| = \frac{\sqrt{\lambda+W^\mathrm{h}(\gamma(h(t)))}}{\varepsilon} \leq \frac{\sqrt{\lambda+m_\gamma}}{\varepsilon},
    \end{equation}
    for a.e. $t\in(-\tau,\tau)$, and
    \begin{equation}\label{eq:reparametrization_estim}
        \int_{-\tau}^\tau \left[ \frac{1}{\varepsilon} W^\mathrm{h} (\gamma \circ h (t)) + \varepsilon |(\gamma\circ h)'(t)|^2 \right] \dd t \leq \int_{-1}^1 2 \sqrt{W^\mathrm{h} (\gamma (s))} |\gamma'(s)| \dd s + 2 \sqrt{\lambda}\, L(\gamma).
    \end{equation}
\end{lemma}
 
\begin{proof}
    Define
    $$
    2\tau \coloneqq \int_{-1}^1 \frac{\varepsilon L(\gamma)}{2\sqrt{\lambda+W^\mathrm{h}(\gamma(s))}}\dd s,\quad\quad
    \theta(r) \coloneqq -\tau + \int_{-1}^r \frac{\varepsilon L(\gamma)}{2\sqrt{\lambda+W^\mathrm{h}(\gamma(s))}}\dd s.
    $$
    Since $0\leq W^\mathrm{h}(\gamma(s))\leq m_\gamma$, we get the bounds on $\tau$. Since $W^\mathrm{h}$ is continuous by Lemma \ref{lemma:continuity_homogenized} and $\gamma$ is continuous, we have $\theta\in C^1([-1,1])$ with
    $$ \theta'(r) = \frac{\varepsilon L(\gamma)}{2\sqrt{\lambda+W^\mathrm{h}(\gamma(r))}} \geq \frac{\varepsilon L(\gamma)}{2\sqrt{\lambda+m_\gamma}} > 0 \qquad\text{for all } r\in[-1,1], $$
    and $\theta$ maps $[-1,1]$ onto $[-\tau,\tau]$. Therefore, $h\coloneqq\theta^{-1}\in C^1([-\tau,\tau])$ is strictly increasing, $h(\pm\tau)=\pm1$, and
    \begin{equation}\label{eq:reparametrization_hprime}
        h'(t) = \frac{2\sqrt{\lambda+W^\mathrm{h}(\gamma(h(t)))}}{\varepsilon L(\gamma)} \qquad\text{for all } t\in[-\tau,\tau].
    \end{equation}
    Let $Z\subset(-1,1)$ be the set of points where $\gamma$ is not differentiable or $|\gamma'|\neq L(\gamma)/2$, which has measure zero. Since $h^{-1}=\theta$ is Lipschitz, the set $h^{-1}(Z)=\theta(Z)$ has measure zero as well. For every $t\in(-\tau,\tau)\setminus\theta(Z)$ we have $|\gamma'(h(t))|=L(\gamma)/2$, which, together with \eqref{eq:reparametrization_hprime}, gives \eqref{eq:reparametrization_diffeq}. Moreover, at such points $\gamma$ is differentiable at $h(t)$, so the chain rule gives \eqref{eq:reparametrization_lip}. In particular, $\gamma\circ h$ is Lipschitz.
 
    Finally, by \eqref{eq:reparametrization_diffeq}, the integrand on the left-hand side of \eqref{eq:reparametrization_estim} is equal to $(2W^\mathrm{h}(\gamma(h(t)))+\lambda)/\varepsilon$ for a.e. $t$. Using the change of variables $t=\theta(s)$, together with
    $$ \theta'(s)=\frac{\varepsilon|\gamma'(s)|}{\sqrt{\lambda+W^\mathrm{h}(\gamma(s))}} \qquad\text{for a.e. } s\in(-1,1), $$
    we get
    \begin{align*}
        \int_{-\tau}^\tau \left[ \frac{1}{\varepsilon} W^\mathrm{h} (\gamma (h(t))) + \varepsilon | \gamma'(h(t)) |^2 (h'(t))^2 \right] \dd t
        &= \int_{-1}^1 \frac{2W^\mathrm{h}(\gamma(s))+\lambda}{\sqrt{\lambda+W^\mathrm{h}(\gamma(s))}}\,|\gamma'(s)|\dd s \\
        &\leq \int_{-1}^1 2\sqrt{\lambda+W^\mathrm{h}(\gamma(s))}\,|\gamma'(s)|\dd s \\
        &\leq \int_{-1}^1 2\sqrt{W^\mathrm{h}(\gamma(s))}\,|\gamma'(s)|\dd s + 2\sqrt\lambda\, L(\gamma),
    \end{align*}
    where we used the sub-additivity of the square root.
\end{proof}

\begin{proof}[Proof of Proposition~\ref{prop:limsup}.]
\textbf{Step 1: Reduction to $C^2$ interfaces.} Let $A \coloneqq \{ x \in \Omega : u(x) = a \}$. Note that since $u \in \mathrm{BV} (\Omega; \{a,b\})$, it follows that $A$ is a set of finite perimeter.
Using Proposition \ref{prop:set_approx} and a diagonalization argument, without loss of generality we prove the result for $u \in \mathrm{BV} (\Omega; \{a,b\})$ such that $\partial A \cap \Omega$ is of class $C^2$ and such that 
$$\mathcal{H}^{N-1}(\partial A \cap \partial \Omega) = 0. $$

\textbf{Step 2: Construction of recovery sequence.} Let $u \in \mathrm{BV} (\Omega; \{a,b\})$ be as in Step 0. Fix $\varsigma > 0$, and let $\gamma \in \mathrm{Lip} ([-1,1]; \R^M)$ with $\gamma(-1)= a$, $\gamma(1)=b$, be such that
\begin{equation}\label{eq:limsup_quasigeodesic}
    \int_{-1}^1 2 \sqrt{W^\mathrm{h} (\gamma(s))} | \gamma'(s) | \dd s \leq \sigma^\mathrm{h} + \varsigma.
\end{equation}

Since $a\neq b$, the length $L(\gamma)$ of $\gamma$ (see Remark \ref{rem:constant_speed}) is positive. Thanks to Remark \ref{rem:constant_speed}, applied with $V=W^\mathrm{h}$, we can assume that $|\gamma'(s)| = L(\gamma)/2$ for a.e. $s\in(-1,1)$. Therefore, we can apply Lemma \ref{lemma:modica_mortola} with
\begin{equation*}
    \varepsilon \coloneqq \varepsilon_n, \qquad \lambda \coloneqq \left( \frac{\varsigma}{L(\gamma)} \right)^2.
\end{equation*}
Therefore, for each $n \in \N$, we have $\tau_n$ and $h_n \in C^1 ([-\tau_n,\tau_n]; [-1,1])$ such that \eqref{eq:reparametrization_diffeq}, \eqref{eq:reparametrization_estim}, and \eqref{eq:reparametrization_lip} hold. Let now $\mathrm{s} \colon \R^N \to \R$ be the signed distance function from $\partial A$. Since $\partial A$ is of class $C^2$, by Lemma \ref{lemma:tubular} we know that $\mathrm{s}(x)$ is $1$-Lipschitz. Therefore, for every $n \in \N$ we define $u_n \in W^{1,2}(\Omega; \R^M)$ as
\begin{gather}\label{eq:def_un}
    u_n (x) \coloneqq \begin{cases}
        a \qquad &\mathrm{s}(x) < - \tau_n, \\
        \gamma (h_n (\mathrm{s}(x))) \quad  &|\mathrm{s}(x) | \leq \tau_n, \\
        b \qquad &\mathrm{s}(x) > \tau_n.
    \end{cases}
\end{gather}

By Lemma \ref{lemma:modica_mortola}, we have $c_1 \varepsilon_n \leq \tau_n \leq c_2 \varepsilon_n$, where
$$ c_1 \coloneqq \frac{L(\gamma)}{2\sqrt{\lambda+m_\gamma}}, \qquad c_2 \coloneqq \frac{L(\gamma)}{2\sqrt{\lambda}}, \qquad m_\gamma \coloneqq \max_{s\in[-1,1]}W^\mathrm{h}(\gamma(s)). $$
Therefore, $\tau_n \to 0$ as $n \to \infty$. This implies that $u_n \to u$ strongly in $L^1 (\Omega; \R^M)$.

We define 
\begin{equation}\label{eq:limsup_interface}
    A_n \coloneqq \{ x \in \Omega : | \mathrm{s}(x) | \leq \tau_n \} \implies |A_n| \leq c \tau_n.
\end{equation}

\textbf{Step 3: Negligible contribution of microstructure.} We claim that
\begin{equation}\label{eq:limsup_claim}
    \limsup_{n \to \infty} \int_{A_n} \frac{1}{\varepsilon_n} \left[ W \Bigg( \!\bigg\{\! \frac{x}{\delta_n} \! \bigg\}_{Q_1}\!, \bigg\{\! \frac{\delta_n}{\eta_n}\!\bigg\{ \!\frac{x}{\delta_n} \!\bigg\}_{Q_1} \!\bigg\}_{Q_2}\!, u_n(x) \!\Bigg) - W^\mathrm{h} (u_n(x)) \right] \dd x \leq 0.
\end{equation}

We first note that $u_n$ is Lipschitz continuous with constant of order $\varepsilon_n^{-1}$. Indeed, let $v_n\colon\R\to\R^M$ be defined as
$$ v_n(t) \coloneqq \begin{cases}
    a \qquad & t<-\tau_n,\\
    \gamma(h_n(t)) \quad & |t|\leq\tau_n,\\
    b \qquad & t>\tau_n,
\end{cases} $$
so that $u_n = v_n\circ\mathrm{s}$. Since $h_n(\pm \tau_n) = \pm 1$, the function $v_n$ is continuous, and, by \eqref{eq:reparametrization_lip}, it is Lipschitz with constant $\sqrt{\lambda+m_\gamma}/\varepsilon_n$. Since $\mathrm{s}(x)$ is Lipschitz with constant $1$ thanks to Lemma \ref{lemma:tubular}, we get that \eqref{eq:def_un} defines a Lipschitz function in $\R^N$, with
\begin{equation}\label{eq:limsup_LipschitzU}
    | u_n (x) - u_n (x') | \leq \frac{C}{\varepsilon_n} | x-x'| \qquad \text{for all } x,x' \in \R^N,
\end{equation}
where $C\coloneqq\sqrt{\lambda+m_\gamma}$ depends only on $\varsigma$ and $\gamma$, and not on $n$.

Let now 
$$ I_n \coloneqq \left\{ \xi_1 \in \Xi_1 : \delta_n \left( \xi_1 + Q_1 \right) \cap A_n \neq \emptyset \right\}. $$
Define $S \coloneqq \max \{ |a|,|b|,\| \gamma\|_\infty \}$. Since $|u_n| \leq S$, by \ref{itm:1_H4} we have that $W \leq C_S$. This, together with $| \Lambda_\delta | \leq C \delta_n$, gives us
\begin{align}
    &\int_{A_n} W \Bigg( \!\bigg\{\! \frac{x}{\delta_n} \! \bigg\}_{Q_1}\!, \bigg\{\! \frac{\delta_n}{\eta_n}\!\bigg\{ \!\frac{x}{\delta_n} \!\bigg\}_{Q_1} \!\bigg\}_{Q_2}\!, u_n(x) \!\Bigg) \dd x \nonumber \\
    &\hspace{1cm}\leq C_S C \delta_n + \sum_{\xi_1 \in I_n} \int_{\delta_n (\xi_1 + Q_1)} W \Bigg( \!\bigg\{\! \frac{x}{\delta_n} \! \bigg\}_{Q_1}\!, \bigg\{\! \frac{\delta_n}{\eta_n}\!\bigg\{ \!\frac{x}{\delta_n} \!\bigg\}_{Q_1} \!\bigg\}_{Q_2}\!, u_n(x) \!\Bigg) \dd x.
\end{align}
Fix now $\xi_1 \in I_n$ and let $D \coloneqq \delta_n (\xi_1 + Q_1)$. Using the same changes of variables as in the proof of Lemma \ref{lemma:needed}, and using again $W \leq C_S$ together with $|\Lambda_{1,\eta}| \leq C \frac{\eta}{\delta}$, we get
\begin{align}
    &\int_D W \Bigg( \!\bigg\{\! \frac{x}{\delta_n} \! \bigg\}_{Q_1}\!, \bigg\{\! \frac{\delta_n}{\eta_n}\!\bigg\{ \!\frac{x}{\delta_n} \!\bigg\}_{Q_1} \!\bigg\}_{Q_2}\!, u_n(x) \!\Bigg) \dd x\nonumber\\
    &\hspace{1cm}\leq C \delta_n^N \frac{\eta_n}{\delta_n} + \delta_n^N \int_{\widehat{Q}_{1,\eta}} \int_{Q_2} W(g_n(y_1,y_2),y_2,p_n(y_1,y_2)) \dd y_2 \dd y_1, \label{eq:limsup_intermediate_sum}
\end{align}
where $g_n(y_1,y_2)$ is given in Remark \ref{rem:H1}, and $p_n (y_1,y_2) \coloneqq u_n (\delta_n \xi_1 + \delta_n g_n (y_1,y_2))$. We also define $q_n (y_1,y_2) \coloneqq u_n (\delta_n \xi_1)$.
Using \eqref{eq:limsup_LipschitzU} we have
$$ | p_n - q_n | = | u_n (\delta_n \xi_1 + \delta_n g_n) - u_n (\delta_n \xi_1) | \leq C \frac{\delta_n}{\varepsilon_n}. $$
We can therefore apply Lemma \ref{lemma:uniform_replacement} to get
\begin{align*}
    \int_{\widehat{Q}_{1,\eta}} \int_{Q_2} W(g_n,y_2,p_n) \dd y_2 \dd y_1 &\leq \int_{\widehat{Q}_{1,\eta}} \int_{Q_2} W(y_1,y_2,q_n) \dd y_2 \dd y_1 + \omega_S \left( \frac{\eta_n}{\delta_n} + C \frac{\delta_n}{\varepsilon_n} \right) \\
    &\leq W^\mathrm{h} (q_n) + \omega_S \left( \frac{\eta_n}{\delta_n} + C \frac{\delta_n}{\varepsilon_n} \right).
\end{align*}
Using again \eqref{eq:limsup_LipschitzU}, we have that for all $x \in D = \delta_n (\xi_1 + Q_1)$ it holds
$$ | u_n (x) - q_n| = | u_n (x) - u_n (\delta_n \xi_1)| \leq C \frac{\delta_n}{\varepsilon_n}. $$
Therefore, denoting by $\omega_S^\mathrm{h}$ the modulus of continuity of $W^\mathrm{h}$ in $\overline{B_S}$, we get
$$ \int_D W^\mathrm{h}(u_n(x))\dd x \geq \delta_n^N \left( W^\mathrm{h}(q_n) - \omega_S^\mathrm{h} \left( C \frac{\delta_n}{\varepsilon_n} \right) \right). $$
Putting together the last three estimates, we obtain
\begin{align}
    &\int_D W \Bigg( \!\bigg\{\! \frac{x}{\delta_n} \! \bigg\}_{Q_1}\!, \bigg\{\! \frac{\delta_n}{\eta_n}\!\bigg\{ \!\frac{x}{\delta_n} \!\bigg\}_{Q_1} \!\bigg\}_{Q_2}\!, u_n(x) \!\Bigg) \dd x\nonumber \\
    &\qquad\leq \int_D W^\mathrm{h} (u_n(x)) \dd x + \delta_n^N \omega_S \left( \frac{\eta_n}{\delta_n} + C \frac{\delta_n}{\varepsilon_n} \right) + \delta_n^N \omega_S^\mathrm{h} \left( C \frac{\delta_n}{\varepsilon_n} \right) + C \delta_n^N \frac{\eta_n}{\delta_n}. \label{eq:limsup_summing}
\end{align}
Since, by definition of $I_n$, we have
$$ \delta_n^N \# I_n \leq C \tau_n + C \delta_n \leq C \varepsilon_n, $$
we can sum \eqref{eq:limsup_summing} over $\xi \in I_n$, and using \eqref{eq:limsup_intermediate_sum} we get
\begin{align*}
    &\int_{A_n} W \Bigg( \!\bigg\{\! \frac{x}{\delta_n} \! \bigg\}_{Q_1}\!, \bigg\{\! \frac{\delta_n}{\eta_n}\!\bigg\{ \!\frac{x}{\delta_n} \!\bigg\}_{Q_1} \!\bigg\}_{Q_2}\!, u_n(x) \!\Bigg) \dd x \\
    &\qquad\leq \int_{A_n} W^\mathrm{h}(u_n(x)) \dd x + C \delta_n + C \varepsilon_n \omega_S \left( \frac{\eta_n}{\delta_n} + C \frac{\delta_n}{\varepsilon_n} \right) + C \varepsilon_n \omega_S^\mathrm{h} \left( C \frac{\delta_n}{\varepsilon_n} \right) + C \varepsilon_n \frac{\eta_n}{\delta_n}.
\end{align*}
Rearranging and dividing by $\varepsilon_n$ we get
\begin{align*}
    &\int_{A_n} \frac{1}{\varepsilon_n} \left[ W \Bigg( \!\bigg\{\! \frac{x}{\delta_n} \! \bigg\}_{Q_1}\!, \bigg\{\! \frac{\delta_n}{\eta_n}\!\bigg\{ \!\frac{x}{\delta_n} \!\bigg\}_{Q_1} \!\bigg\}_{Q_2}\!, u_n(x) \!\Bigg) - W^\mathrm{h} (u_n(x)) \right] \dd x \\
    &\qquad\leq C \frac{\delta_n}{\varepsilon_n} + C \omega_S \left( \frac{\eta_n}{\delta_n} + C \frac{\delta_n}{\varepsilon_n} \right) + C \omega_S^\mathrm{h} \left( C \frac{\delta_n}{\varepsilon_n} \right) + C \frac{\eta_n}{\delta_n} \to 0,
\end{align*}
which is exactly \eqref{eq:limsup_claim}.\\

    \textbf{Step 4: Convergence of energy.}
    Our objective now is to estimate the value of $F_n (u_n)$, which, using the definition of the recovery sequence, is equal to
    \begin{align*}
        F_n (u_n) &= \int_\Omega \left[ \frac{1}{\varepsilon_n} W \Bigg( \!\bigg\{\! \frac{x}{\delta_n} \! \bigg\}_{Q_1}\!, \bigg\{\! \frac{\delta_n}{\eta_n}\!\bigg\{ \!\frac{x}{\delta_n} \!\bigg\}_{Q_1} \!\bigg\}_{Q_2}\!, u_n \!\Bigg) + \varepsilon_n | \nabla u_n |^2 \right] \dd x \\
        &= \int_{A_n} \left[ \frac{1}{\varepsilon_n} W \Bigg( \!\bigg\{\! \frac{x}{\delta_n} \! \bigg\}_{Q_1}\!, \bigg\{\! \frac{\delta_n}{\eta_n}\!\bigg\{ \!\frac{x}{\delta_n} \!\bigg\}_{Q_1} \!\bigg\}_{Q_2}\!, u_n \!\Bigg) + \varepsilon_n | \nabla u_n |^2 \right] \dd x.
    \end{align*}
    By adding and subtracting $\frac{1}{\varepsilon_n} W^\mathrm{h}(u_n)$ inside the integral we get
    \begin{align}
        \int_{A_n} &\left[ \frac{1}{\varepsilon_n} W \Bigg( \!\bigg\{\! \frac{x}{\delta_n} \! \bigg\}_{Q_1}\!, \bigg\{\! \frac{\delta_n}{\eta_n}\!\bigg\{ \!\frac{x}{\delta_n} \!\bigg\}_{Q_1} \!\bigg\}_{Q_2}\!, u_n \!\Bigg) + \varepsilon_n | \nabla u_n |^2 \right] \dd x \nonumber \\
        &\qquad\qquad= \int_{A_n} \left[ \cfrac{1}{\varepsilon_n} W^\mathrm{h} (u_n) + \varepsilon_n | \nabla u_n |^2 \right] \dd x \nonumber \\
        &\qquad\qquad\qquad\qquad+ \cfrac{1}{\varepsilon_n} \int_{A_n} \left[ W \Bigg( \!\bigg\{\! \frac{x}{\delta_n} \! \bigg\}_{Q_1}\!, \bigg\{\! \frac{\delta_n}{\eta_n}\!\bigg\{ \!\frac{x}{\delta_n} \!\bigg\}_{Q_1} \!\bigg\}_{Q_2}\!, u_n \!\Bigg) - W^\mathrm{h} (u_n) \right] \dd x. \label{eq:limsup_intermediate}
    \end{align}
    Thanks to Step 2, this last term is non-positive in the limit, therefore we get
    \begin{equation}\label{eq:limsup_last_step} \limsup_{n\to \infty} F_n (u_n) \leq \limsup_{n \to \infty} \int_{A_n} \left[ \frac{1}{\varepsilon_n} W^\mathrm{h} \left( u_n \right) + \varepsilon_n | \nabla u_n |^2 \right] \dd x.
    \end{equation}
    We want to obtain an upper bound for the right-hand side.
    We know that, with $v_n$ as in Step 2, it holds
    \[
    \nabla u_n(x) = v_n'(\mathrm{s}(x))\otimes\nabla\mathrm{s}(x)
    \]
    for a.e. $x\in\Omega$, and that
    \[
    |v_n'(t)| = |\gamma'(h_n(t))|\,|h_n'(t)|
    \]
    for a.e. $|t|<\tau_n$, by \eqref{eq:reparametrization_lip}. Indeed, $v_n$ is differentiable outside a set $Z_n\subset\R$ of measure zero. Since $|\mathrm{D}\mathrm{s}|=1$ a.e., the coarea formula gives
    $$ |\mathrm{s}^{-1}(Z_n)\cap\Omega| = \int_{\mathrm{s}^{-1}(Z_n)\cap\Omega}|\mathrm{D}\mathrm{s}(x)|\dd x = \int_{Z_n}\mathcal H^{N-1}(\mathrm{S}_t\cap\Omega)\dd t = 0, $$
    and at all the other points of $\Omega$ where $\mathrm{s}$ is differentiable, the classical chain rule applies.
    Then, from \eqref{eq:limsup_last_step}, using the definition of $S_t$ from Lemma \ref{lemma:tubular}, \eqref{eq:def_un}, the coarea formula and \eqref{eq:hausdorff_approx}, we get
    \begin{align*}
        \limsup_{n \to \infty}& F_n (u_n) \leq \limsup_{n \to \infty} \int_{A_n} \left[ \frac{1}{\varepsilon_n} W^\mathrm{h} \left( u_n \right) + \varepsilon_n | \nabla u_n |^2 \right] \dd x \\
        &= \limsup_{n \to \infty} \int_{-\tau_n}^{\tau_n} \left[ \frac{1}{\varepsilon_n} W^\mathrm{h} (\gamma(h_n(s))) + \varepsilon_n | \gamma'(h_n(s)) |^2 | h_n' (s) |^2 \right] \mathcal{H}^{N-1} \left( \mathrm{S}_s \cap \Omega \right) \dd s \\
        &\leq \limsup_{n \to \infty} \sup_{|s|\leq \tau_n} \mathcal{H}^{N-1} \left( \mathrm{S}_s \cap \Omega \right) \int_{-\tau_n}^{\tau_n} \left[ \frac{1}{\varepsilon_n} W^\mathrm{h} (\gamma(h_n(s))) + \varepsilon_n | \gamma'(h_n(s)) |^2 | h_n' (s) |^2 \right] \dd s \\
        &\leq \limsup_{n \to \infty} \sup_{|s|\leq \tau_n} \mathcal{H}^{N-1} \left( \mathrm{S}_s \cap \Omega \right) \left[ \int_{-1}^1 2 \sqrt{W^\mathrm{h} (\gamma(t))} | \gamma' (t)|\dd t + 2 \sqrt{\lambda} L(\gamma) \right] \\
        &\leq ( \sigma^\mathrm{h} + 3 \varsigma ) \limsup_{n \to \infty} \sup_{|s|\leq \tau_n} \mathcal{H}^{N-1} \left( \mathrm{S}_s \cap \Omega \right) \\
        &= ( \sigma^\mathrm{h} + 3 \varsigma ) \mathcal{H}^{N-1} (\partial A \cap \Omega),
    \end{align*}
    where in the last inequality we used Lemma \ref{lemma:modica_mortola}, \eqref{eq:limsup_quasigeodesic}, and Lemma \ref{lemma:tubular}.
    
    We conclude by arbitrariness of $\varsigma$.
\end{proof}

\section{Mass-constrained case}\label{sec:mass}

The goal of this section is to prove Theorem \ref{thm:mass_constraint}.
In the case of a mass constraint, the proof needs to change slightly. The procedures for the compactness and liminf inequality remain unchanged, and only the proof for the limsup has to be modified.\\
\begin{proof}
\textbf{Step 1: Correction of the set.} We now have $u \in \mathrm{BV}(\Omega; \{ a,b\})$ with 
$$ \int_\Omega u \dd x = m \in \R^M. $$
By defining as before $A \coloneqq \{ u = a \}$, we get that $u = a \mathbbm{1}_A + b (1 - \mathbbm{1}_A)$. Substituting in the mass constraint we get that
$$ (a-b)|A| = m - b|\Omega|. $$
This suggests that in order for the target mass $m$ to be reachable by the constraint, we need $m$ to be contained in the segment connecting the points $a|\Omega|$ and $b|\Omega|$, i.e.
$$ m = t a  |\Omega| + (1-t)b|\Omega|, \qquad \text{for some }\, t \in [0,1]. $$
Substituting this back gives us
$$ |A| = t |\Omega|, \qquad \text{for some }\, t \in [0,1], $$
which is perfectly well-defined. The edge cases give $A = \emptyset$ when $t = 0$ and $A = \Omega$ when $t = 1$, and therefore we can study the interesting case when $t \in (0,1)$. From now on, for simplicity, we will define $\mu \coloneqq t |\Omega| \in \R$.

In Proposition \ref{prop:set_approx} we approximated $A$ with a sequence $\{A_n\}_n$ of appropriate $C^2$ sets. The problem here is that we did not require these sets to satisfy the constraint $|A_n|=\mu$ for every $n \in \N$. Therefore, this is the first change that needs to be addressed: we thus follow an idea originally due to Ryan Murray, appropriately modified for our case (see \cite{murray2016slow}). 

\begin{figure}
    \centering
    \input{fig_mass_constraint_set}
    \caption{Idea of the correction of the $C^2$ approximation.}
    \label{fig:mass_constraint_set}
\end{figure}

Let us fix two points $x_1, x_2 \in \partial^* A$ in the reduced boundary of $A$. Let us choose then $k \in \N$ and define 
$$ D_k \coloneqq \left( A \cup B \! \left( \! x_1, \frac{1}{k} \! \right) \right) \setminus B \! \left( \! x_2, \frac{1}{k} \! \right). $$
It is possible to see that
$$ \mathbbm{1}_{D_k} \to \mathbbm{1}_A \quad \text{in } L^1(\Omega; \R) \quad \text{as } k \to \infty. $$
Moreover, we can check that
\begin{align*}
	\left| \mathrm{D} \mathbbm{1}_{D_k} \right| (\Omega) &= \mathcal{H}^{N-1} \left(\Omega \cap \partial^* D_k \right) \\
	&\leq \mathcal{H}^{N-1} \left(\Omega \cap \partial^* A \right) + \mathcal{H}^{N-1} \left(\Omega \cap \partial B \! \left( \! x_1, \frac{1}{k} \! \right) \right) + \mathcal{H}^{N-1} \left(\Omega \cap \partial B \! \left( \! x_2, \frac{1}{k} \! \right) \right) \\
	&\to \mathcal{H}^{N-1} \left(\Omega \cap \partial^* A \right) \\
	&= \left| \mathrm{D} \mathbbm{1}_{A} \right| (\Omega).
\end{align*}
Therefore we also have
$$ \lim_{k \to \infty} \mathrm{Per}(D_k;\Omega) = \mathrm{Per}(A;\Omega). $$
Since $x_1, x_2 \in \partial^* A$, by definition of reduced boundary they have density $\frac{1}{2}$ with respect to $A$, that is
$$ \lim_{r \to 0} \frac{|A \cap B(x_1,r)|}{|B(x_1,r)|} = \cfrac{1}{2}, \qquad \lim_{r \to 0} \frac{|A \cap B(x_2,r)|}{|B(x_2,r)|} = \cfrac{1}{2}. $$
Therefore, this implies that there exists $k$ large enough such that
$$ \left| A \cap B \left( x_1, \frac{1}{k} \right) \right| > \frac{1}{4} \left| B \left( x_1, \frac{1}{k} \right) \right|, \qquad \left| A \cap B \left( x_2, \frac{1}{k} \right) \right| < \frac{3}{4} \left| B \left( x_2, \frac{1}{k} \right) \right|. $$
Let now $\{D_k^n\}_n$ be the sequence of $C^2$ sets obtained by applying Proposition \ref{prop:set_approx} to the set $D_k$. Since we have convergence in perimeter and $L^1$, this implies that for a fixed $k$, there exists $\widetilde{n}_1(k) \in \N$ such that for every $n \geq \widetilde{n}_1(k)$ we have
$$ \left| \mathrm{Per} \left( D_k; \Omega \right) - \mathrm{Per} \left( D_k^n; \Omega \right) \right| \leq \frac{1}{k}, \qquad \int_\Omega \left| \mathbbm{1}_{D_k} (x) - \mathbbm{1}_{D_k^n}(x) \right| \dd x \leq \frac{1}{k}. $$
Since these sets $D_k^n$ are obtained by using a standard mollifying procedure and taking a super-level set, we know that there exists $\widetilde{n}_2(k) \in \N$ such that for every $n \geq \max \{ \widetilde{n}_1(k), \widetilde{n}_2(k) \}$ we have both the previous inequalities and also
$$ B \left( x_1, \left( \frac{4}{5} \right)^{\frac{1}{N}} \frac{1}{k} \right) \subset D_k^n, \qquad B \left( x_2, \left( \frac{4}{5} \right)^{\frac{1}{N}} \frac{1}{k} \right) \subset \Omega \setminus D_k^n. $$
This holds because 
$$ \left( \frac{4}{5} \right)^{\frac{1}{N}} < 1 \quad \forall N \in \N. $$

We now have three cases to distinguish. If $| D_k^n | = \mu$ then we do not have to do anything for now. Assume now that $ | D_k^n | > \mu $. Define $r_k^n >0$ to be the radius such that
$$ \left| B \left( x_1, r_k^n \right) \right| = \left| D_k^n \right| - \mu > 0, $$
and define (see Fig. \ref{fig:mass_constraint_set}) 
$$ A_k^n \coloneqq D_k^n \setminus B \left( x_1, r_k^n \right). $$
We now want to prove that $$ r_k^n < \left( \frac{4}{5} \right)^\frac{1}{N} \frac{1}{k}. $$
Since we know that $|A|=\mu$ and
$$ \left| A \cap B \left( x_1, \frac{1}{k} \right) \right| > \frac{1}{4} \left| B \left( x_1, \frac{1}{k} \right) \right|, $$
we have
\begin{align*}
	\left| D_k \right| &= |A| + \left| B \left( x_1, \frac{1}{k} \right) \setminus A \right| - \left| B \left( x_2, \frac{1}{k} \right) \cap A \right| \\
	&\leq |A| + \left| B \left( x_1, \frac{1}{k} \right) \right| - \left| B \left( x_1, \frac{1}{k} \right) \cap A \right| - \left| B \left( x_2, \frac{1}{k} \right) \cap A \right| \\
	&\leq |A| + \left| B \left( x_1, \frac{1}{k} \right) \right| - \frac{1}{4} \left| B \left( x_1, \frac{1}{k} \right) \right| \\
	&= |A| + \frac{3}{4} \left| B \left( x_1, \frac{1}{k} \right) \right| \\
	&= \mu + \left| B \left( x_1, \left( \frac{3}{4} \right)^\frac{1}{N} \frac{1}{k} \right) \right|.
\end{align*}
Therefore, for $n$ large enough we also get 
$$ \left| B \left( x_1, r_k^n \right) \right| = \left| D_k^n \right| - \mu \leq \left| B \left( x_1, \left( \frac{3}{4} \right)^\frac{1}{N} \frac{1}{k} \right) \right| < \left| B \left( x_1, \left( \frac{4}{5} \right)^\frac{1}{N} \frac{1}{k} \right) \right|, $$
giving the desired estimate on $r_k^n$.
This in particular implies that the set $A_k^n$ is also $C^2$, and it follows that
$$ |A_k^n| = |D_k^n| - \left| B \left( x_1, r_k^n \right) \right| = |D_k^n| - |D_k^n| + \mu = \mu. $$
Assume now that $|D_k^n| < \mu$. Define $r_k^n >0$ to be the radius such that
$$ \left| B \left( x_2, r_k^n \right) \right| = \mu - \left| D_k^n \right| > 0, $$
and define 
$$ A_k^n \coloneqq D_k^n \cup B \left( x_2, r_k^n \right). $$
We now want to prove that $$ r_k^n < \left( \frac{4}{5} \right)^\frac{1}{N} \frac{1}{k}. $$
Since we know that $|A|=\mu$ and
$$ \left| A \cap B \left( x_2, \frac{1}{k} \right) \right| < \frac{3}{4} \left| B \left( x_2, \frac{1}{k} \right) \right|, $$
we have
\begin{align*}
	\left| D_k \right| &= |A| + \left| B \left( x_1, \frac{1}{k} \right) \setminus A \right| - \left| B \left( x_2, \frac{1}{k} \right) \cap A \right| \\
	&\geq |A| - \frac{3}{4} \left| B \left( x_2, \frac{1}{k} \right) \right| \\
	&= \mu - \left| B \left( x_2, \left( \frac{3}{4} \right)^\frac{1}{N} \frac{1}{k} \right) \right|.
\end{align*}
Therefore, for $n$ large enough we also get 
$$ \left| B \left( x_2, r_k^n \right) \right| = \mu - \left| D_k^n \right| \leq \left| B \left( x_2, \left( \frac{3}{4} \right)^\frac{1}{N} \frac{1}{k} \right) \right| < \left| B \left( x_2, \left( \frac{4}{5} \right)^\frac{1}{N} \frac{1}{k} \right) \right|, $$
giving the desired estimate on $r_k^n$.
This in particular implies that the set $A_k^n$ is also $C^2$, and it follows that
$$ |A_k^n| = |D_k^n| + \left| B \left( x_2, r_k^n \right) \right| = |D_k^n| + \mu - |D_k^n| = \mu. $$
Therefore we just proved that for every $k$ we can modify the sequence of sets obtained from Proposition \ref{prop:set_approx} applied to $D_k$ such that every element of the sequence satisfies the mass constraint. Using a diagonal argument, we can therefore obtain the desired conclusion.\\

\textbf{Step 2: Correction of the sequence.} By Step 1, by the lower semi-continuity of the $\Gamma$-limsup with respect to the $L^1$ convergence, and since the approximating sets satisfy the constraint, it is enough to construct a recovery sequence for $u = a \mathbbm{1}_A + b \mathbbm{1}_{\Omega \setminus A}$, where $\partial A \cap \Omega$ is of class $C^2$, with $\mathcal{H}^{N-1}(\partial A \cap \partial \Omega) = 0$ and $|A| = t |\Omega|$, $t \in (0,1)$.

Let $u_n$ be defined as in \eqref{eq:def_un}, and let
$$ m_n \coloneqq \int_\Omega u_n (x) \dd x \in \R^M. $$
If, for a given $n \in \N$, we have $m_n = m$, then we are done. Let us suppose that $m_n \neq m$. 
Let us recall the definition of $A_n$:
$$ A_n \coloneqq \{ x \in \Omega : | \mathrm{dist}(x, \partial A) | \leq \tau_n \}, $$
which is the set where $u_n(x) \notin \{a, b\}$. 

Let $S \coloneqq \max \{ |a|,|b|,\| \gamma \|_\infty \}$. Since $u_n = u$ in $\Omega \setminus A_n$ and $|u_n - u| \leq 2S$ in $A_n$, we have
\begin{equation}\label{eq:mass_difference}
    | m_n - m | \leq \int_{A_n} |u_n - u| \dd x \leq 2S |A_n| \leq C \varepsilon_n,
\end{equation}
for some constant $C > 0$ independent of $n$. The last inequality follows from \eqref{eq:limsup_interface}, together with $\tau_n\leq c_2\varepsilon_n$ (see Step 1 of the proof of Proposition \ref{prop:limsup}).

Let $x_0 \in \Omega \setminus A_n$ such that $u_n (x_0) = u(x_0) = a$, and fix $r > 0$ small enough such that $u_n \equiv a$ on $B \coloneqq B(x_0,r)$. Fix now $\phi \in C_c^\infty (B; [0,1])$ with integral over $B$ equal to $1$, and define
$$ w_n (x) \coloneqq (m - m_n) \phi (x) \in \R^M, \qquad v_n \coloneqq u_n + w_n, $$
(see Figure \ref{fig:mass_constraint_recovery}).

\begin{figure}
    \centering
    \input{fig_mass_constraint_recovery}
    \caption{Idea of the correction of the recovery sequence.}
    \label{fig:mass_constraint_recovery}
\end{figure}

It follows that $v_n \in W^{1,2}(\Omega; \R^M)$, and
$$ \int_\Omega v_n \dd x = \int_\Omega u_n \dd x + \int_\Omega w_n \dd x = m_n + (m - m_n) \int_B \phi \dd x = m. $$
Moreover, by \eqref{eq:mass_difference}, we have
\begin{equation}\label{eq:mass_inequalities}
    \| v_n - u_n \|_{L^1} = | m_n - m| \leq C \varepsilon_n, \qquad \| w_n \|_\infty \leq C \varepsilon_n, \qquad \| \nabla w_n \|_\infty \leq C \| \nabla \phi \|_\infty \varepsilon_n.
\end{equation}
In particular, $v_n \to u$ strongly in $L^1 (\Omega; \R^M)$.

We now estimate the energy. By definition we have
$$ F_n (v_n) = F_n (u_n) + \int_B \left[ \frac{1}{\varepsilon_n} W \left( \!\bigg\{\! \frac{x}{\delta_n} \! \bigg\}_{Q_1}\!, \bigg\{\! \frac{\delta_n}{\eta_n}\!\bigg\{ \!\frac{x}{\delta_n} \!\bigg\}_{Q_1} \!\bigg\}_{Q_2}\!, a + w_n(x) \!\right) + \varepsilon_n | \nabla w_n(x) |^2 \right] \dd x. $$
Let $n$ be large enough such that $C \varepsilon_n \leq r_0$, where $r_0$ is given by \ref{itm:1_H6}. Using \ref{itm:1_H6}, the monotonicity of $\omega_a$, and \eqref{eq:mass_inequalities}, we get
\begin{align*}
    \int_B \frac{1}{\varepsilon_n} W \left( \!\bigg\{\! \frac{x}{\delta_n} \! \bigg\}_{Q_1}\!, \bigg\{\! \frac{\delta_n}{\eta_n}\!\bigg\{ \!\frac{x}{\delta_n} \!\bigg\}_{Q_1} \!\bigg\}_{Q_2}\!, a + w_n(x) \!\right) \dd x &\leq \frac{1}{\varepsilon_n} \omega_a (C \varepsilon_n) \int_B | w_n (x) | \dd x\\
    &\leq C \omega_a (C \varepsilon_n) \to 0.
\end{align*}
For the gradient term, \eqref{eq:mass_inequalities} gives
$$ \int_B \varepsilon_n | \nabla w_n |^2 \dd x \leq C \| \nabla \phi \|^2_\infty \varepsilon_n^3 \to 0.$$
As both terms vanish as $n \to \infty$, using Proposition \ref{prop:limsup} we get
$$ \limsup_{n \to \infty} \widehat{F}_n (v_n) = \limsup_{n \to \infty} F_n (u_n) \leq F_\infty (u) = \widehat{F}_\infty (u), $$
and we conclude.
\end{proof}

\begin{remark}
    The above construction works for every $N,M \geq 1$, under the assumption of uniform differentiability in the wells.
\end{remark}

\begin{remark}\label{rem:diff}
	The assumption of uniform differentiability can be substituted by a stronger one, namely that close enough to the wells it holds $W(y_1,y_2,z) \leq |z-a|^\beta$, for a $\beta > 1$. This assumption implies uniform differentiability, but the contrary does not hold. 
    
    Moreover, pointwise differentiability is not enough for the mass constraint: indeed, let $W$ be defined in a neighborhood of $a$ as
    $$ W(y_1,y_2,z) = |z-a| \min \left\{ \frac{|z-a|}{|y_1|}, 1 \right\}. $$
    For every $y_1 \neq 0$, this function is continuous in $(y_1,z)$ and differentiable in $z$ at $a$ with zero differential, but 
    $$ \sup_{y_1 \in Q_1} \frac{W(y_1,y_2,z)}{|z-a|} = 1, $$
    for all $z \neq a$ close to $a$. Therefore, the best we can have is $W(y_1,y_2,z) \leq |z-a|$, which is not enough for proving that the energy of the perturbation vanishes.
\end{remark}


\section*{Acknowledgments}

Both authors were supported by the grant NWO-OCENW.M.21.336, MATHEMAMI - Mathematical Analysis of phase
Transitions in HEterogeneous MAterials with Materials Inclusions.

\vspace{0.5cm}
\noindent\textbf{Data Availability}: Data sharing not applicable to this article as no datasets were generated or analysed during the current study.

\vspace{0.5cm}
\noindent\textbf{Ethics Statement}: The authors have no conflicts of interest to declare relevant to this article.

\bibliographystyle{siam}
\bibliography{ref}
            
\end{document}

%% file: Packages.tex
\usepackage[latin1]{inputenc}
\usepackage[T1]{fontenc}
\usepackage[english]{babel}
\usepackage{csquotes}
\usepackage{latexsym,amsmath,amssymb,amsfonts}
\usepackage{dsfont}
\usepackage{color}
\usepackage{graphicx}
\usepackage{float}
\usepackage{esint}
\usepackage{bm}
\usepackage{mathtools}
\usepackage[top=1in, bottom=1.25in, left=1in, right=1in]{geometry}
\usepackage{enumitem}
\usepackage{hyperref}
\usepackage{cleveref}
\usepackage{bbm}
\usepackage{tikz}                        
\usetikzlibrary{calc}
\usepackage{pgfplots}    
\usepackage{soul}         
\usepackage{marginnote}

\hypersetup{
  colorlinks = true,
  linkcolor = blue,
  citecolor = red
}

\theoremstyle{plain}
\begingroup
\newtheorem{theorem}{Theorem}[section]
\newtheorem{lemma}[theorem]{Lemma}
\newtheorem{proposition}[theorem]{Proposition}
\newtheorem{corollary}[theorem]{Corollary}
\endgroup
\theoremstyle{definition}
\begingroup
\newtheorem{definition}[theorem]{Definition}
\newtheorem{remark}[theorem]{Remark}

\endgroup
\theoremstyle{remark}

\makeatletter
\def\namedlabel#1#2{\begingroup
    #2%
    \def\@currentlabel{#2}%
    \phantomsection\label{#1}\endgroup
}
\makeatother

%% file: Commands.tex
\newcommand{\N}{\mathbb N}

\newcommand{\R}{\mathbb R}

\renewcommand{\S}{\mathbb S}

\newcommand{\ca}{\mathbb{1}}

\renewcommand{\l}{{\mathcal L}}
\renewcommand{\ln}{{\mathcal{L}^N}}

\newcommand{\dd}{\,\mathrm{d}}
\renewcommand{\l}{\,\dd{\mathcal L}}

\newcommand{\e}{\varepsilon}
\newcommand{\p}{\varphi}

\def\ca{\mathbbmss{1}}

\newcommand{\f}{\mathcal{F}}

\newcommand{\average}{{\mathchoice {\kern1ex\vcenter{\hrule height.4pt
width 6pt
depth0pt} \kern-9.7pt} {\kern1ex\vcenter{\hrule height.4pt width 4.3pt
depth0pt}
\kern-7pt} {} {} }}
\newcommand{\med}{\average\int}

\def\XXint#1#2#3{{\setbox0=\hbox{$#1{#2#3}{\int}$ }
\vcenter{\hbox{$#2#3$ }}\kern-.6\wd0}}

\allowdisplaybreaks

%% file: figstyle.tex
\usepackage{amsmath,amssymb}
\usepackage{tikz}
\usetikzlibrary{calc,arrows.meta,backgrounds}

\definecolor{matone}{RGB}{243,214,210}   
\definecolor{inclone}{RGB}{236,128,114}  
\definecolor{mattwo}{RGB}{214,240,206}   
\definecolor{incltwo}{RGB}{108,190,92}   
\definecolor{cellA}{RGB}{198,226,240}    
\definecolor{cellB}{RGB}{176,190,236}    
\definecolor{phaseA}{RGB}{206,234,240}   
\definecolor{layer}{RGB}{243,214,210}    
\definecolor{accent}{RGB}{176,38,32}     
\definecolor{curveg}{RGB}{20,140,120}    

\tikzset{
  bdry/.style   ={line width=0.7pt},
  grid/.style   ={line width=0.35pt, black!75},
  fine/.style   ={line width=0.25pt, black!60},
  dim/.style    ={{Latex[length=1.6mm,width=1.2mm]}-{Latex[length=1.6mm,width=1.2mm]}, line width=0.45pt},
  vec/.style    ={-{Latex[length=1.8mm,width=1.3mm]}, line width=0.6pt},
  zoom/.style   ={dashed, dash pattern=on 2.2pt off 1.8pt, line width=0.35pt, black!60},
  pointer/.style={-{Latex[length=1.6mm,width=1.2mm]}, line width=0.55pt, accent},
  dot/.style    ={circle, fill=accent, inner sep=0pt, minimum size=3.2pt},
}

\def\omegablob{plot[domain=0:360, samples=180, smooth cycle]
  ({2.75*(1+0.10*cos(3*\x+25)+0.05*sin(2*\x-30))*cos(\x)},
   {2.45*(1+0.10*cos(3*\x+25)+0.05*sin(2*\x-30))*sin(\x)})}
\tikzset{
  egg/.pic={
    \begin{scope}[rotate around={-45:(0.5,0.5)}]
      \path[pic actions] plot[domain=0:360, samples=90, smooth cycle]
        ({0.5+0.27*cos(\x)*(1-0.13*sin(\x))}, {0.5+0.33*sin(\x)});
    \end{scope}
  },
}
\definecolor{acctwo}{RGB}{40,120,40}     
\tikzset{pointertwo/.style={-{Latex[length=1.6mm,width=1.2mm]}, line width=0.55pt, acctwo}}

%% file: fig_one_scale.tex
\begin{tikzpicture}

\begin{scope}
  \begin{scope}
    \clip \omegablob;
    \begin{scope}[rotate=45, scale=1.0]
      \fill[matone] (0,-1) rectangle ++(1,1);          
      \foreach \i in {-5,...,4} \foreach \j in {-5,...,4}
        \pic[transform shape, fill=inclone, draw=black, line width=0.35pt] at (\i,\j) {egg};
      \foreach \k in {-5,...,5} {
        \draw[grid] (\k,-5) -- (\k,5);
        \draw[grid] (-5,\k) -- (5,\k);
      }
      \coordinate (ctop) at (1,0);
      \coordinate (cbot) at (0,-1);
      \coordinate (cleft) at (0,0);
    \end{scope}
  \end{scope}
  \draw[bdry] \omegablob;
  \node at (-2.65,2.25) {$\Omega$};
  \begin{scope}[rotate=45]
    \draw[dim] (2.95,0.85) -- ++(0,-1) node[midway, anchor=south west, inner sep=1.5pt] {$\delta$};
  \end{scope}
  \node (dQ) at (-1.0,-2.95) {$\delta Q_1$};
  \draw[pointer] (dQ.east) to[bend right=35] (0.55,-0.45);
\end{scope}

\begin{scope}[shift={(8.7,0)}]
  \begin{scope}[rotate=45, scale=3.1, shift={(-0.5,-0.5)}]
    \fill[matone] (0,0) rectangle (1,1);
    \pic[transform shape, fill=inclone, draw=black, line width=0.5pt] at (0,0) {egg};
    \draw[bdry] (0,0) rectangle (1,1);
    \coordinate (Qtop) at (1,1);
    \coordinate (Qbot) at (0,0);
    \draw[dim] (1.12,1) -- (1.12,0) node[midway, anchor=south west, inner sep=2pt] {$1$};
  \end{scope}
  \node at (-1.1,-1.45) {$Q_1$};
\end{scope}

\draw[zoom] (ctop) -- (Qtop);
\draw[zoom] (cbot) -- (Qbot);

\end{tikzpicture}

%% file: fig_two_scales.tex
\begin{tikzpicture}
\def\ea{0.38}\def\eb{0.24}\def\h{0.125}

\begin{scope}[scale=0.92]
  \begin{scope}
    \clip \omegablob;
    \begin{scope}[rotate=45]
      \fill[matone] (0,-1) rectangle ++(1,1);
      \coordinate (ctop) at (1,0);
      \coordinate (cbot) at (0,-1);
      \coordinate (cc) at (0.5,-0.5);
    \end{scope}
    \begin{scope}[rotate=45]
      \foreach \i in {-5,...,4} \foreach \j in {-5,...,4}
        \coordinate (e-\i-\j) at (\i+0.5,\j+0.5);
    \end{scope}
    \foreach \i in {-5,...,4} \foreach \j in {-5,...,4}
      \fill[inclone, draw=black, line width=0.35pt] (e-\i-\j) ellipse ({\ea} and {\eb});
    \begin{scope}
      \clip (cc) ellipse ({\ea} and {\eb});
      \fill[incltwo] (cc) circle (0.035);
      \foreach \k in {-4,...,4} {
        \draw[fine] ($(cc)+(\k*\h-\h/2,-1)$) -- ++(0,2);
        \draw[fine] ($(cc)+(-1,\k*\h-\h/2)$) -- ++(2,0);
      }
    \end{scope}
    \draw[black, line width=0.35pt] (cc) ellipse ({\ea} and {\eb});
    \draw[acctwo, line width=0.6pt] ($(cc)+(-\h/2,-\h/2)$) rectangle ++(\h,\h);
    \coordinate (sc) at (cc);
    \begin{scope}[rotate=45]
      \foreach \k in {-5,...,5} {
        \draw[grid] (\k,-5) -- (\k,5);
        \draw[grid] (-5,\k) -- (5,\k);
      }
    \end{scope}
  \end{scope}
  \draw[bdry] \omegablob;
  \node at (-2.65,2.25) {$\Omega$};
  \begin{scope}[rotate=45]
    \draw[dim] (2.95,0.85) -- ++(0,-1) node[midway, anchor=south west, inner sep=1.5pt] {$\delta$};
  \end{scope}
  \node (dQ) at (-1.0,-2.95) {$\delta Q_1$};
  \draw[pointer] (dQ.east) to[bend right=35] ($(cc)+(-0.05,-0.4)$);
  \node (eQ) at (1.45,-2.55) {$\eta Q_2$};
  \draw[pointertwo] (eQ.north) to[bend right=30] ($(cc)+(0.03,-0.07)$);
\end{scope}

\begin{scope}[shift={(6.85,0)}]
  \def\s{3.0}          
  \pgfmathsetmacro{\H}{\h*\s} 
  \begin{scope}[rotate=45, scale=\s, shift={(-0.5,-0.5)}]
    \fill[matone] (0,0) rectangle (1,1);
    \coordinate (Qtop) at (1,1);
    \coordinate (Qbot) at (0,0);
  \end{scope}
  \fill[inclone] (0,0) ellipse ({\ea*\s} and {\eb*\s});
  \begin{scope}
    \clip (0,0) ellipse ({\ea*\s} and {\eb*\s});
    \foreach \k in {-4,...,4} \foreach \l in {-3,...,3}
      \fill[incltwo] (\k*\H,\l*\H) circle ({0.28*\H});
    \foreach \k in {-4,...,4} {
      \draw[fine] (\k*\H-\H/2,-2) -- ++(0,4);
      \draw[fine] (-2,\k*\H-\H/2) -- ++(4,0);
    }
  \end{scope}
  \draw[line width=0.5pt] (0,0) ellipse ({\ea*\s} and {\eb*\s});
  \draw[acctwo, line width=1pt] (\H/2,-\H/2) rectangle ++(\H,\H);
  \coordinate (qne) at (1.5*\H,\H/2);
  \coordinate (qse) at (1.5*\H,-\H/2);
  \begin{scope}[rotate=45, scale=\s, shift={(-0.5,-0.5)}]
    \draw[bdry] (0,0) rectangle (1,1);
    \draw[dim] (1.12,1) -- (1.12,0) node[midway, anchor=south west, inner sep=2pt] {$1$};
  \end{scope}
  \draw[{Latex[length=1.1mm,width=1mm]}-{Latex[length=1.1mm,width=1mm]}, line width=0.4pt]
    (\H/2,{\eb*\s+0.17}) -- ++(\H,0) node[midway, above=1pt] {$\frac{\eta}{\delta}$};
  \draw[fine, dotted] (\H/2,\H/2) -- (\H/2,{\eb*\s+0.17});
  \draw[fine, dotted] (1.5*\H,\H/2) -- (1.5*\H,{\eb*\s+0.17});
  \node at (-1.25,-1.4) {$Q_1$};
  \node (eQm) at (1.25,-1.75) {$\frac{\eta}{\delta}Q_2$};
  \draw[pointertwo] (eQm.north west) to[bend left=25] (\H,-0.12);
\end{scope}

\begin{scope}[shift={(11.35,0)}]
  \def\t{2.3}
  \fill[inclone] (-\t/2,-\t/2) rectangle (\t/2,\t/2);
  \fill[incltwo, draw=black, line width=0.5pt] (0,0) circle ({0.28*\t});
  \draw[bdry] (-\t/2,-\t/2) rectangle (\t/2,\t/2);
  \coordinate (Pnw) at (-\t/2,\t/2);
  \coordinate (Psw) at (-\t/2,-\t/2);
  \draw[dim] (-\t/2,\t/2+0.25) -- (\t/2,\t/2+0.25) node[midway, above=1pt] {$1$};
  \node at (0,-\t/2-0.35) {$Q_2$};
\end{scope}

\draw[zoom] (ctop) -- (Qtop);
\draw[zoom] (cbot) -- (Qbot);
\draw[zoom] (qne) -- (Pnw);
\draw[zoom] (qse) -- (Psw);

\end{tikzpicture}

%% file: fig_first_regime.tex
\begin{tikzpicture}
\def\W{5.4}   
\def\Ht{4.6}  
\def\e{2.5}   
\def\d{1.0}   
\def\n{4}     
\pgfmathsetmacro{\et}{\d/\n}
\pgfmathsetmacro{\yc}{\Ht/2}

\draw[bdry, rounded corners=0.9cm] (0,0) rectangle (\W,\Ht);
\node at (0.55,\Ht-0.5) {$\Omega$};

\draw[zoom] (0,\yc-\e/2) -- (\W,\yc-\e/2);
\draw[zoom] (0,\yc+\e/2) -- (\W,\yc+\e/2);

\pgfmathsetmacro{\xo}{(\W-5*\d)/2}
\fill[matone] (\xo+2*\d,\yc-\d/2) rectangle ++(\d,\d);
\begin{scope}
  \clip (\xo+2*\d,\yc-\d/2) rectangle ++(\d,\d);
  \fill[mattwo] (\xo+2*\d+2*\et,\yc-\d/2+2*\et) rectangle ++(\et,\et);
  \foreach \k in {1,...,3} {
    \draw[fine] (\xo+2*\d+\k*\et,\yc-\d/2) -- ++(0,\d);
    \draw[fine] (\xo+2*\d,\yc-\d/2+\k*\et) -- ++(\d,0);
  }
\end{scope}
\draw[acctwo, line width=0.6pt] (\xo+2*\d+2*\et,\yc-\d/2+2*\et) rectangle ++(\et,\et);
\foreach \k in {0,...,5} \draw[grid] (\xo+\k*\d,\yc-\d/2) -- ++(0,\d);
\draw[grid] (0,\yc-\d/2) -- (\W,\yc-\d/2);
\draw[grid] (0,\yc+\d/2) -- (\W,\yc+\d/2);
\draw[line width=0.6pt] (\xo+2*\d,\yc-\d/2) rectangle ++(\d,\d);

\draw[dim] (\W+0.35,\yc-\e/2) -- (\W+0.35,\yc+\e/2) node[midway, right=1pt] {$\varepsilon$};
\draw[dim] (-0.35,\yc-\d/2) -- (-0.35,\yc+\d/2) node[midway, left=1pt] {$\delta$};

\node (dQ) at (\xo+0.75*\d,\yc-\e/2-0.45) {$\delta Q_1$};
\draw[pointer] (dQ.north) to[bend left=25] (\xo+2*\d+0.18,\yc-\d/2+0.12);
\node (eQ) at (\xo+3.6*\d,\yc-\e/2-0.45) {$\eta Q_2$};
\draw[pointertwo] (eQ.north) to[bend right=25] (\xo+2*\d+2.5*\et,\yc-\d/2+2.3*\et);
\end{tikzpicture}

%% file: fig_unfolding.tex
\begin{tikzpicture}

\begin{scope}
  \def\ax{1.5}\def\bx{0.32}\def\by{1.1}
  \begin{scope}[shift={(0.55,0.35)}, x={(\ax cm,0)}, y={(\bx cm,\by cm)}]
    \fill[cellA] (0,0) -- (1,0) -- (1,1) -- (0,1) -- cycle;
    \fill[cellB] (2,2) -- (3,2) -- (3,3) -- (2,3) -- cycle;
    \begin{scope}
      \clip[shift={(-0.55,-0.35)}, x={(1cm,0)}, y={(0,1cm)}] (0,0) rectangle (6.9,4.0);
      \foreach \k in {-1,...,4} \draw[grid] (-3,\k) -- (6,\k);
      \foreach \k in {-3,...,6} \draw[grid] (\k,-1) -- (\k,5);
    \end{scope}
    \coordinate (O)  at (0,0);
    \coordinate (P)  at (2,2);
    \coordinate (x)  at (2.72,2.62);
    \coordinate (fx) at (0.72,0.62);
    \draw[vec] (O) -- (fx);
    \draw[vec] (P) -- (x);
    \node[dot] at (O) {};
    \node[dot] at (P) {};
    \node[dot] at (x) {};
    \node[dot] at (fx) {};
    \node[anchor=north east, inner sep=2pt] at (O) {$0$};
    \node[right] at (x) {$x$};
    \node[anchor=north west, inner sep=1.5pt] at (P) {$\delta\big\lfloor \tfrac{x}{\delta}\big\rfloor_{Q_1}$};
    \node[anchor=west] (fl) at (1.12,0.28) {$\delta\big\{\tfrac{x}{\delta}\big\}_{Q_1}$};
    \draw[pointer] (fl.west) to[bend left=20] ($(O)!0.6!(fx)$);
    \node[anchor=east] at (0.03,0.8) {$\delta Q_1$};
  \end{scope}
  \node[fill=white, inner sep=1pt] at (0.3,3.75) {$\Omega$};
\end{scope}

\begin{scope}[shift={(10.0,2.0)}]
  \def\blob{plot[domain=0:360, samples=180, smooth cycle]
      ({2.35*(1+0.11*cos(3*\x-40)+0.06*sin(2*\x))*cos(\x)},
       {2.35*(1+0.11*cos(3*\x-40)+0.06*sin(2*\x))*sin(\x)})}
  \begin{scope}
    \clip \blob;
    \begin{scope}[rotate=45, scale=0.62]
      \foreach \i/\j in {-4/-1,-3/-2,-3/-1,-3/0,-2/-3,-2/-2,-2/-1,-2/0,-2/1,-1/-3,-1/-2,-1/-1,-1/0,-1/1,-1/2,0/-3,0/-2,0/-1,0/0,0/1,0/2,1/-3,1/-2,1/-1,1/0,1/1,1/2,2/-2,2/-1,2/0}
        \fill[matone] (\i,\j) rectangle ++(1,1);
      \foreach \k in {-7,...,7} {
        \draw[grid] (\k,-7) -- (\k,7);
        \draw[grid] (-7,\k) -- (7,\k);
      }
    \end{scope}
  \end{scope}
  \draw[bdry] \blob;
  \node at (-2.25,2.15) {$\Omega$};
  \node (Oh) at (2.75,2.2) {$\widehat{\Omega}_\delta$};
  \draw[pointer] (Oh.south west) to[bend right=15] (1.05,1.0);
  \node (La) at (-2.95,-1.75) {$\Lambda_\delta$};
  \draw[pointer] (La.north) to[bend left=25] (-1.8,-0.47);
\end{scope}

\end{tikzpicture}

%% file: fig_substitution.tex
\begin{tikzpicture}
\def\R{2.97}     
\def\H{0.28}     
\def\ea{1.7}\def\eb{1.1}   
\def\c{0.25}   

\begin{scope}
  \clip (\R,0) -- (0,\R) -- (-\R,0) -- (0,-\R) -- cycle;
  \fill[matone] (-\R,-\R) rectangle (\R,\R);
  \fill[inclone] (0,0) ellipse ({\ea} and {\eb});
  \draw[black, opacity=0.2, line width=2*\c cm] (0,0) ellipse ({\ea} and {\eb});
  \foreach \k in {-11,...,10} \foreach \l in {-11,...,10} {
    \pgfmathsetmacro{\m}{max(abs(\k*\H)+abs(\l*\H), abs((\k+1)*\H)+abs(\l*\H),
                             abs(\k*\H)+abs((\l+1)*\H), abs((\k+1)*\H)+abs((\l+1)*\H))}
    \ifdim\m pt>\R pt \fill[black, opacity=0.2] (\k*\H,\l*\H) rectangle ++(\H,\H); \fi
  }
  \foreach \k in {-11,...,11} {
    \draw[fine] (\k*\H,-\R) -- (\k*\H,\R);
    \draw[fine] (-\R,\k*\H) -- (\R,\k*\H);
  }
  \draw[line width=0.5pt] (0,0) ellipse ({\ea} and {\eb});
\end{scope}
\draw[bdry] (\R,0) -- (0,\R) -- (-\R,0) -- (0,-\R) -- cycle;

\coordinate (good) at (-0.84,-0.28);   
\coordinate (bad)  at (1.40,0.28);     
\draw[acctwo, line width=0.8pt] (good) rectangle ++(\H,\H);
\draw[accent, line width=0.8pt] (bad) rectangle ++(\H,\H);

\node[fill=inclone, inner sep=1pt] at (-0.42,0.42) {$E_1$};
\node[fill=matone, inner sep=1pt] at (-0.42,1.96) {$E_2$};
\node at (-2.35,-1.35) {$Q_1$};
\node (L) at (-3.5,1.75) {$\Lambda_{1,\eta}$};
\draw[pointer, black] (L.east) to[bend left=20] (-1.62,1.2);
\node (T) at (-3.35,-0.55) {$T_1\cup T_2$};
\draw[pointer, black] (T.north east) to[bend left=20] (-1.62,-0.25);

\def\s{2.1}                    
\pgfmathsetmacro{\z}{\s/\H}    
\foreach \name/\bx/\by/\cx/\cy/\col in {good/-0.84/-0.28/5.6/1.35/acctwo, bad/1.40/0.28/5.6/-1.55/accent} {
  \begin{scope}[shift={(\cx-\s/2,\cy-\s/2)}]
    \begin{scope}
      \clip (0,0) rectangle (\s,\s);
      \begin{scope}[scale=\z, shift={(-\bx,-\by)}]
        \fill[matone] (\bx-1,\by-1) rectangle ++(3,3);
        \fill[inclone] (0,0) ellipse ({\ea} and {\eb});
        \draw[line width=0.5pt] (0,0) ellipse ({\ea} and {\eb});
      \end{scope}
    \end{scope}
    \draw[\col, line width=0.8pt] (0,0) rectangle (\s,\s);
    \coordinate (\name-sw) at (0,0);
    \coordinate (\name-nw) at (0,\s);
  \end{scope}
}
\draw[zoom] ($(good)+(\H,\H)$) -- (good-nw);
\draw[zoom] ($(good)+(\H,0)$) -- (good-sw);
\draw[zoom] ($(bad)+(\H,\H)$) -- (bad-nw);
\draw[zoom] ($(bad)+(\H,0)$) -- (bad-sw);

\begin{scope}[shift={(5.6-\s/2,1.35-\s/2)}, x=\s cm, y=\s cm]
  \coordinate (y) at (0.22,0.72); \coordinate (g) at (0.78,0.25);
  \draw[line width=0.45pt, dashed] (y) -- (g);
  \node[circle, fill=black, inner sep=0pt, minimum size=3.2pt] at (y) {};
  \node[circle, fill=black, inner sep=0pt, minimum size=3.2pt] at (g) {};
  \node[anchor=south west, inner sep=1pt] at (y) {$y_1$};
  \node[anchor=north east, inner sep=1pt] at (g) {$g_n$};
  \node[anchor=west, align=left, font=\small] at (1.08,0.5) {$y_1\in E_1\setminus T_1$:\\[1pt] $g_n(y_1,y_2)\in E_1$};
\end{scope}
\begin{scope}[shift={(5.6-\s/2,-1.55-\s/2)}, x=\s cm, y=\s cm]
  \coordinate (y) at (0.15,0.45); \coordinate (g) at (0.88,0.62);
  \draw[line width=0.45pt, dashed] (y) -- (g);
  \node[circle, fill=black, inner sep=0pt, minimum size=3.2pt] at (y) {};
  \node[circle, fill=black, inner sep=0pt, minimum size=3.2pt] at (g) {};
  \node[anchor=north, inner sep=2pt] at (y) {$y_1$};
  \node[anchor=south, inner sep=2pt] at (g) {$g_n$};
  \node[anchor=west, align=left, font=\small] at (1.08,0.5) {$y_1\in T_1$:\\[1pt] $g_n(y_1,y_2)\in E_2$};
  \draw[dim] (0,-0.1) -- (1,-0.1) node[midway, below=1pt] {$\frac{\eta}{\delta}$};
\end{scope}
\end{tikzpicture}

%% file: fig_curve_cutting.tex
\begin{tikzpicture}

\begin{scope}[x=1.45cm, y=1.7cm]
  \def\t{0.09}
  \pgfmathsetmacro{\st}{sqrt(\t)}
  \def\Wh(#1){((#1)^2-1)^2*(1+0.25*(#1))}
  \draw[-{Latex[length=1.6mm,width=1.2mm]}, line width=0.45pt] (-1.75,0) -- (1.8,0);
  \draw[-{Latex[length=1.6mm,width=1.2mm]}, line width=0.45pt] (-1.75,-0.05) -- (-1.75,1.45);
  \begin{scope}
    \clip (-1.75,-0.2) rectangle (1.8,1.3);
    \draw[line width=0.6pt] plot[domain=-1.5:1.5, samples=150, smooth] (\x,{\Wh(\x)});
    \draw[curveg, line width=1.1pt] plot[domain=-1.5:1.5, samples=300] (\x,{max(\Wh(\x)-\t,0)});
  \end{scope}
  \draw[line width=0.35pt, dotted] (-1.75,\t) -- (1.6,\t);
  \draw[line width=0.35pt, dotted] (-1.75,\st) -- (1.6,\st);
  \node[anchor=east, inner sep=2pt] at (-1.75,\t) {$t$};
  \node[anchor=east, inner sep=2pt] at (-1.75,\st) {$\sqrt{t}$};
  \foreach \lo/\hi in {-1.291/-0.635, 0.703/1.216}
    \draw[accent, line width=1.6pt] (\lo,0) -- (\hi,0);
  \node[below=3pt] at (-1,0) {$a$};
  \node[below=3pt] at (1,0) {$b$};
  \node[anchor=west] at (0.6,1.15) {$V$};
  \node[anchor=east, curveg] at (-0.45,1.02) {$V_t$};
  \node[accent, font=\small] at (0.02,-0.33) {$\{V\le\sqrt{t}\}$};
\end{scope}

\begin{scope}[shift={(5.95,-0.2)}]
  \def\rh{0.85}
  \coordinate (a) at (0,0.2);
  \coordinate (b) at (5.4,0.55);
  \fill[black!6] (a) circle (\rh);
  \fill[black!6] (b) circle (\rh);
  \draw[line width=0.4pt, dashed] (a) circle (\rh);
  \draw[line width=0.4pt, dashed] (b) circle (\rh);
  \coordinate (sa) at ($(a)+(18:\rh)$);
  \coordinate (sb) at ($(b)+(158:\rh)$);
  \draw[line width=0.6pt, black!45]
    (a) to[out=135, in=250] ($(a)+(-0.45,0.5)$)
        to[out=70, in=120] ($(a)+(0.35,1.3)$)
        to[out=-60, in=200] ($(a)+(0.45,-0.3)$)
        to[out=20, in=198] (sa);
  \draw[line width=0.6pt, black!45]
    (sb) to[out=-40, in=200] ($(b)+(-0.2,-0.35)$)
         to[out=20, in=300] ($(b)+(0.5,1.25)$)
         to[out=120, in=70] ($(b)+(-0.25,0.45)$)
         to[out=-110, in=130] (b);
  \draw[line width=1.1pt]
    (sa) to[out=18, in=200] (2.1,1.45) to[out=20, in=160] (3.6,1.2) to[out=-20, in=140] (sb);
  \draw[accent, line width=1.1pt] (a) -- (sa);
  \draw[accent, line width=1.1pt] (sb) -- (b);
  \foreach \p in {a,b,sa,sb} \node[circle, fill=black, inner sep=0pt, minimum size=3.4pt] at (\p) {};
  \node[anchor=north east, inner sep=1.5pt] at (a) {$a$};
  \node[anchor=north west, inner sep=1.5pt] at (b) {$b$};
  \node[anchor=north west, inner sep=1pt] at (sa) {$\gamma(s_a)$};
  \node[anchor=north east, inner sep=1pt] at (sb) {$\gamma(s_b)$};
  \node[anchor=north] at ($(a)+(0,-\rh)$) {$B(a,\rho)$};
  \node[anchor=north] at ($(b)+(0,-\rh)$) {$B(b,\rho)$};
  \node[anchor=south, align=center, font=\small] at (2.8,1.55) {$V_t\ge(1-\sqrt{t})\,V$\\ along $\gamma|_{[s_a,s_b]}$};
\end{scope}
\end{tikzpicture}

%% file: fig_mass_constraint_set.tex
\begin{tikzpicture}
\def\W{8.6}\def\Ht{5.2}
\def\rho{0.85}  
\def\f{0.28}    
\def\r{0.42}    
\coordinate (x1) at (2.9,1.45);
\coordinate (x2) at (5.15,3.7);

\begin{scope}
  \clip (0,0) rectangle (\W,\Ht);
  \fill[phaseA, even odd rule]
    (1.45,0)
    -- ($(x1)+(225:\rho+\f)$) to[out=45, in=305] ($(x1)+(210:\rho)$)
    arc[start angle=210, end angle=60, radius=\rho]
    to[out=-30, in=225] ($(x1)+(45:\rho+\f)$)
    -- ($(x2)+(225:\rho+\f)$) to[out=45, in=150] ($(x2)+(240:\rho)$)
    arc[start angle=240, end angle=390, radius=\rho]
    to[out=120, in=225] ($(x2)+(45:\rho+\f)$)
    -- (6.85,5.4) -- (\W+1,\Ht+1) -- (\W+1,-1) -- (1.45,-1) -- cycle
    (x1) circle (\r);
  \draw[line width=0.4pt, black!60] (1.45,0) -- (\W,\W-1.45);
  \draw[line width=0.4pt, black!60] (x1) circle (\rho);
  \draw[line width=0.4pt, black!60] (x2) circle (\rho);
  \draw[curveg, line width=0.8pt]
    (1.45,0)
    -- ($(x1)+(225:\rho+\f)$) to[out=45, in=305] ($(x1)+(210:\rho)$)
    arc[start angle=210, end angle=60, radius=\rho]
    to[out=-30, in=225] ($(x1)+(45:\rho+\f)$)
    -- ($(x2)+(225:\rho+\f)$) to[out=45, in=150] ($(x2)+(240:\rho)$)
    arc[start angle=240, end angle=390, radius=\rho]
    to[out=120, in=225] ($(x2)+(45:\rho+\f)$)
    -- (6.85,5.4);
  \draw[curveg, line width=0.8pt] (x1) circle (\r);
\end{scope}
\draw[line width=0.4pt, black!50] (0,0) rectangle (\W,\Ht);

\node[dot] at (x1) {};
\node[dot] at (x2) {};
\node[anchor=west, inner sep=1pt] at ($(x1)+(0.05,-0.13)$) {$x_1$};
\node[anchor=west, inner sep=1pt] at ($(x2)+(0.05,-0.13)$) {$x_2$};
\node at (6.9,1.5) {$A^n_k$};
\node[anchor=south east] at (1.85,0.35) {$\partial A$};
\node[anchor=east] at ($(x1)+(160:\rho)$) {$B\big(x_1,\tfrac1k\big)$};
\node[anchor=east] at ($(x2)+(150:\rho)$) {$B\big(x_2,\tfrac1k\big)$};
\end{tikzpicture}

%% file: fig_mass_constraint_recovery.tex
\begin{tikzpicture}
\def\W{6.4}\def\Ht{4.4}
\def\tau{0.28}   
\def\r{0.55}     
\coordinate (x0) at (4.85,1.55);

\begin{scope}
  \clip (0,0) rectangle (\W,\Ht);
  \begin{scope}[shift={(2.0,2.2)}, rotate=-35]
    \fill[phaseA] (0,-8) rectangle (8,8);
    \fill[layer] (-\tau,-8) rectangle (\tau,8);
    \draw[line width=0.4pt, black!60, dash pattern=on 2.5pt off 1.5pt] (0,-8) -- (0,8);
    \draw[line width=0.4pt] (-\tau,-8) -- (-\tau,8);
    \draw[line width=0.4pt] (\tau,-8) -- (\tau,8);
  \end{scope}
  \fill[cellB!70] (x0) circle (\r);
  \draw[line width=0.5pt] (x0) circle (\r);
  \draw[zoom] (x0) circle (2*\r);
  \draw[line width=0.35pt, black!70, dotted] ($(x0)+(-2.2*\r,0)$) -- ($(x0)+(2.2*\r,0)$);
\end{scope}
\draw[line width=0.4pt, black!50] (0,0) rectangle (\W,\Ht);

\node[dot] at (x0) {};
\node[anchor=north, inner sep=2pt] at (x0) {$x_0$};
\draw[vec, line width=0.45pt] (x0) -- ($(x0)+(55:\r)$) node[pos=0.55, anchor=south east, inner sep=1pt] {$r$};
\node[anchor=south east, inner sep=1pt] at ($(x0)+(125:2*\r)$) {$B(x_0,2r)$};
\node at (5.5,3.85) {$\{u_n=a\}$};
\node at (0.95,3.85) {$\{u_n=b\}$};
\node (An) at (0.45,0.55) {$A_n$};
\draw[pointer] (An.north east) to[bend right=25] (1.3,1.25);

\begin{scope}[shift={(9.75,0.75)}]
  \def\s{1.0}   
  \def\hb{1.6}  
  \draw[-{Latex[length=1.6mm,width=1.2mm]}, line width=0.45pt] (-2.3*\s,0) -- (2.4*\s,0);
  \draw[-{Latex[length=1.6mm,width=1.2mm]}, line width=0.45pt] (-2.1*\s,-0.1) -- (-2.1*\s,\hb+0.75);
  \node[anchor=south, inner sep=2pt] at (-2.1*\s,\hb+0.75) {$|w_n|$};
  \fill[cellB!70] plot[domain=-0.995:0.995, samples=120, smooth] ({\x*\s},{\hb*exp(1-1/(1-\x*\x))}) -- cycle;
  \draw[line width=0.7pt] (-2.2*\s,0) -- (-\s,0)
    plot[domain=-0.995:0.995, samples=120, smooth] ({\x*\s},{\hb*exp(1-1/(1-\x*\x))})
    -- (2.2*\s,0);
  \foreach \t/\lab in {-1/{x_0-r},0/{x_0},1/{x_0+r}}
    \draw[line width=0.4pt] (\t*\s,0.06) -- (\t*\s,-0.06) node[below=1pt] {$\lab$};
  \draw[line width=0.35pt, dotted] (-2.1*\s,\hb) -- (0,\hb);
  \node[anchor=east, inner sep=2pt] at (-2.1*\s,\hb) {$C\varepsilon_n$};
\end{scope}
\end{tikzpicture}